\documentclass{amsart}

\usepackage{amssymb}
\usepackage{amsfonts}
\usepackage{amsmath}
\usepackage{amsthm}
\usepackage{graphicx}
\usepackage{mathrsfs}
\usepackage{dsfont}
\usepackage{amscd}
\usepackage{multirow}
\usepackage{palatino}
\usepackage{mathpazo}
\usepackage{aurical}
\usepackage[all]{xy}
\usepackage[scaled]{helvet} 
\usepackage{courier} 
\normalfont
\usepackage[T1]{fontenc}
\usepackage{calligra}
\usepackage{verbatim}
\usepackage[usenames,dvipsnames]{color}
\usepackage[colorlinks=true,linkcolor=Blue,citecolor=Violet]{hyperref}
\usepackage{amsrefs}

\makeatletter
\@namedef{subjclassname@2010}{%
  \textup{2010} Mathematics Subject Classification}
\makeatother

\theoremstyle{plain}
\newtheorem{theorem}{Theorem}[section]

\newtheorem{corollary}[theorem]{Corollary}

\newtheorem{lemma}[theorem]{Lemma}
\newtheorem{proposition}[theorem]{Proposition}

\newtheorem{theorem-definition}[theorem]{Theorem-Definition}

\theoremstyle{definition}
\newtheorem{definition}[theorem]{Definition}

\newtheorem{notation}[theorem]{Notation}

\newtheorem{convention}[theorem]{Convention}

\theoremstyle{remark}

\newtheorem{remark}[theorem]{Remark}

\numberwithin{equation}{section}

\newcommand{\N}{{\mathds{N}}}

\newcommand{\R}{{\mathds{R}}}
\newcommand{\C}{{\mathds{C}}}
\newcommand{\T}{{\mathds{T}}}

\newcommand{\A}{{\mathfrak{A}}}
\newcommand{\B}{{\mathfrak{B}}}

\newcommand{\Lip}{{\mathsf{L}}}

\newcommand{\qpropinquity}[1]{{\mathsf{\Lambda}_{#1}}}

\newcommand{\Kantorovich}[1]{{\mathsf{mk}_{#1}}}

\newcommand{\StateSpace}{{\mathscr{S}}}

\newcommand{\mongekant}{{Mon\-ge-Kan\-to\-ro\-vich metric}}

\newcommand{\unit}{1}

\newcommand{\sa}[1]{{\mathfrak{sa}\left({#1}\right)}}

\newcommand{\dom}[1]{{\operatorname*{dom}\left({#1}\right)}}
\newcommand{\diam}[2]{{\mathrm{diam}\left({#1},{#2}\right)}}

\newcommand{\bridgelength}[2]{{\lambda\left({#1}\middle|{#2}\right)}}

\newcommand{\alg}[1]{{\mathfrak{#1}}}

\renewcommand{\geq}{\geqslant}
\renewcommand{\leq}{\leqslant}

\newcommand{\Latremoliere}{Latr\'{e}moli\`{e}re}
\allowdisplaybreaks[4]

\makeatletter
\newcommand{\vast}{\bBigg@{4}}
\newcommand{\Vast}{\bBigg@{5}}
\makeatother

\begin{document}

\title[Standard homogeneous C*-algebras as compact quantum metric spaces]{Standard homogeneous C*-algebras as compact quantum metric spaces}
\author{Konrad Aguilar}
\address{School of Mathematical and Statistical Sciences \\ Arizona State University \\  901 S. Palm Walk, Tempe, AZ 85287-1804}
\email{konrad.aguilar@asu.edu}
\urladdr{}
\author{Tristan Bice}
\address{Mathematical Institute \\
Polish Academy of Sciences \\
ul. \'{S}niadeckich 8 \\
00-656 Warszawa \\
Poland}
\email{tristan.bice@gmail.com}
\urladdr{}

\date{\today}
\thanks{The first author  was partially supported by the Simons - Foundation grant 346300 and the Polish Government MNiSW 2015-2019 matching fund, and this work is part of the project supported by the grant H2020-MSCA-RISE-2015-691246-QUANTUM DYNAMICS and the Polish Government grant  3542/H2020/2016/2.}
\thanks{Some of this work was completed during the Simons semester hosted at IMPAN during September-December 2016 titled ``Noncommutative Geometry - The Next Generation''.}

\begin{abstract}
Given a compact metric space X and a unital C*-algebra A, we introduce a family of seminorms on the C*-algebra of continuous functions from X to A, denoted C(X, A), induced by classical Lipschitz seminorms that produce compact quantum metrics in the sense of Rieffel if and only if A is finite-dimensional.  As a consequence, we are able isometrically embed X into the state space of C(X,A).  Furthermore, we are able to extend convergence of compact metric spaces in the Gromov-Hausdorff distance to convergence of spaces of matrices over continuous functions on the associated compact metric spaces in \Latremoliere's Gromov-Hausdorff propinquity.
\end{abstract}
\maketitle

\setcounter{tocdepth}{1}
\tableofcontents
\section{Introduction and Background}
Compact quantum metric spaces of  Rieffel \cite{Rieffel98a,Rieffel99,Rieffel05} (motivated by work of Connes \cite{Connes89, Connes})  provide a framework for the study of noncommutative metric geometry just as spectral triples provide for noncommutative differential geometry and C*-algebras provide for noncommutative topology.  This is done by way of quantum metrics induced by  seminorms on C*-algebras that serve as noncommutative analogues to the Lipschitz seminorm on the C*-algebra of complex-valued continuous functions on a compact metric space.    In this article, we will provide families of quantum metrics for the C*-algebra of continuous functions from a compact metric space to a finite dimensional C*-algebra. As an application, we show that  convergence of compact metric spaces in the Gromov-Hausdorff distance can be extended to convergence of spaces of {\em matrices} over continuous functions on the associated compact metric spaces in \Latremoliere's quantum Gromov-Hausdorff propinquity. It has been known since the introduction of Rieffel's quantum distance \cite[Proposition 4.7]{Rieffel00} in 2000 that  convergence of compact metric spaces in the Gromov-Hausdorff distance can be extended to convergence of spaces of continuous functions on the associated compact metric spaces by way of Rieffel's quantum distance (this also holds true for \Latremoliere's quantum Gromov-Hausdorff propinquity \cite[Theorem 6.6]{Latremoliere13}), and  thus, our work answers --- in the positive --- the question of whether this result extends to the matricial case.

We note of another motivation for our work in this article.  The quantum metrics we introduced will be quantum metrics for standard circle algebras, i.e.  $C(\T, \A)$, where $\T$ is the circle and $\A$ is a finite dimensional C*-algebra. Thus,  this article is partly motivated by work of the first author and F. \Latremoliere \ \cite{Aguilar-Latremoliere15, Aguilar16, Aguilar16b}, in which they brought AF algebras into the realm on noncommutative metric geometry by, in part, finding various quantum metrics for AF algebras.  Therefore, following the Elliott classification program \cite{Elliott93b}, a next natural step is to develop quantum metrics for limit circle algebras or A$\T$-algebras or inductive limits of circle algebras \cite[page 159]{Davidson}. But, as done in \cite{Aguilar-Latremoliere15}, it is quite beneficial to first place quantum metrics on the C*-algebras of the inductive sequence that build the inductive limit.  Thus, this article serves to provide a natural family of quantum metrics for circle algebras as a vital step for this pursuit.

Our construction will be based on the quantum metric induced by the Lipschitz seminorm on the C*-algebra of continuos functions on a compact metric space $\left(X, \mathsf{d}_X\right)$, denoted $C(X)$.  Indeed, it is a well-known result (likely due to L. Kantorovich) that the metric on the state space $\StateSpace(C(X))$ defined by:
\begin{equation*}
\Kantorovich{\Lip_{\mathsf{d}_X}}: (\mu, \nu) \in \StateSpace(C(X)) \times \StateSpace (C(X)) \longmapsto \sup_{\Lip_{\mathsf{d}_X}(f) \leq 1} \vert \mu(f)-\nu(f) \vert
\end{equation*}
metrizes the weak* topology of   $\StateSpace(C(X))$, where $\StateSpace(C(X))$ is the state space of $C(X)$ and: 
\begin{equation*}\Lip_{\mathsf{d}_X}(f)=\sup_{x,y\in X, x \neq y} \frac{\vert f(x)-f(y) \vert}{\mathsf{d}_X(x,y)}
\end{equation*} is the Lipschitz constant of $f$, which we call the Lipschitz seminorm.  Furthermore, the map:
\begin{equation*}
\Delta: x \in \left(X, \mathsf{d}_X\right)\longmapsto \delta_x \in  \left(\StateSpace(C(X)), \Kantorovich{\Lip_{\mathsf{d}_X}}\right),
\end{equation*}
 is an isometry onto its image, where $\delta_x: f \in C(X) \longmapsto f(x) \in \C$ is the Dirac point mass at $x$.  Of course, even when $X$ is only compact Hausdorff, this map is a homeomorhism onto its image.  Thus, the contribution of  $\Kantorovich{\Lip_{\mathsf{d}_X}}$ is that this homeomorphism is strengthened to an isometry in the case when $X$ is compact metric.
 (Although these  results about $C(X)$ are well-known, they are difficult to find in the literature, but our results in this paper will provide these results as a corollary, in particular, by Theorem (\ref{fd-thm}) and Corollary (\ref{main-cor})).  
 
 In this article, we claim that our quantum metrics for $C(X, \A)$, when $\A$ is finite-dimensional, will be a suitable generalizations of this setting on $C(X)=C(X, \C)$.  First, the family of seminorms we develop on $C(X, \A)$ will produce quantum metrics if and only if $\A$ is finite dimensional in Section (\ref{cqms-section}).  Next, for a particular choice of seminorm on $C(X, \A)$, we will respresent the above $C(X)$ structure on the unital C*-subalgebra of scalar-valued functions denoted $C(X, \C1_\A)$ where $1_\A$ is the identity of $\A$ (also in Section (\ref{cqms-section})).  Finally,  in Section (\ref{mk-metric-section}), an appropriate analogue to the above  map, $\Delta$, will capture the metric stucture of $X$ in the state space via the quantum metric by way of a bi-Lipschitz map for any of the seminorms we produce, and for particular natural choices of seminorms, the map will be an isometry just as $\Delta$ is in the classical case, and these cases will still be independent of $\A$ and $X$. Furthermore, although some of our main results rest on the finite-dimensionality of $\A$, we note that we are able to prove many crucial algebraic and analytic properties of our seminorms without the assumption of finite-dimensionality on $\A$. And, finally, in Section (\ref{converge-section}), we extend the convergence of compact metric spaces to a purely noncommutative setting by way of matrices over continuous function spaces.  Next, we provide some necessary background for the results of this paper.

\begin{notation}
Let $\A$ be a unital C*-algebra.  Denote the C*-norm of $\A$ by $\Vert \cdot \Vert_\A$ and the unit of $\A$ by $1_\A$.  If $\B$ is a C*-subalgebra of $\A$, then we use $\Vert \cdot \Vert_{\A/\B}$ to denote the quotient norm on $\A/\B$ induced by the C*-norm of $\A$.

Denote the self-adjoint elements of $\A$ by $\sa{\A}.$

Denote the state space of $\A$ by $\StateSpace(\A).$
\end{notation}

\begin{notation}
Let $\A$ be a C*-algebra. Let $\Lip$ be a seminorm defined on $\sa{\A}$.  Its domain is defined as $\dom{\Lip}= \{ a \in \sa{\A} : \Lip (a) < \infty\}$.

If $\Lip$ is defined on $\A$, then we denote $\dom{\Lip}_\A = \{a \in \A : \Lip(a) < \infty \}$ and $\dom{\Lip}=\dom{\Lip}_\A \cap \sa{\A}$.
\end{notation}
\begin{definition}[\cite{Rieffel98a,Rieffel99, Rieffel05}]\label{Monge-Kantorovich-def}
A {\em compact quantum  metric space} $(\A,\Lip)$ is an ordered pair where $\A$ is a unital C*-algebra with unit $1_\A$  and $\Lip$ is a seminorm over $\R$ defined on $\sa{\A}$  whose domain $\dom{\Lip}= \{ a \in \sa{\A} : \Lip (a) < \infty\}$ is a unital dense subspace of $\sa{\A}$ over $\R$ such that:
\begin{enumerate}
\item $\{ a \in \sa{\A} : \Lip(a) = 0 \} = \R\unit_\A$,
\item the \emph{\mongekant} defined, for all two states $\varphi, \psi \in \StateSpace(\A)$, by:
\begin{equation*}
\Kantorovich{\Lip} (\varphi, \psi) = \sup\left\{ |\varphi(a) - \psi(a)| : a\in\dom{\Lip}, \Lip(a) \leq 1 \right\}
\end{equation*}
metrizes the weak* topology of $\StateSpace(\A)$, and
\item the seminorm $\Lip$ is lower semi-continuous with respect to $\|\cdot\|_\A$.
\end{enumerate}
If $(\A,\Lip)$ is a  compact quantum metric space, then we call the seminorm $\Lip$ a {\em Lip-norm}.  
\end{definition}
\begin{remark}\label{mk-metric-dense-remark}
The density condition on $\dom{\Lip}$ in the above definition condition guarantees that the  map $\Kantorovich{\Lip}$ is a metric (possibly taking value $+\infty$) on $\StateSpace(\A)$, which follows by continuity and linearity of states and the fact that every element of a C*-algebra is a linear combination of self-adjoint elements.
\end{remark}

In Rieffel's pioneering work on  compact quantum metric spaces \cite{Rieffel98a}, certain equilavent conditions were given for the requirement that the Monge-Kantorovich metric   metrizes the weak* topology of the state space.  These conditions provide a useful tool for verifying this difficult property.  Further equivalences were given in \cite{Ozawa05}.  The following theorem summarizes all known characterizations of Lip-norms and the proof uses both Arzela-Ascoli theorem and the classical structure of $C(X)$ with $X$ compact metric and its associated Lipschitz seminorm, which, in part, explains the term "compact quantum metric space."

\begin{theorem}[\cite{Rieffel98a,Rieffel99, Ozawa05}]\label{Rieffel-thm}
Let $(\A,\Lip)$ be an ordered pair where $\A$ is unital C*-algebra and $\Lip$ is a lower semi-continuous seminorm defined on $\sa{\A}$ such that its domain $\dom{\Lip}=\{a \in \sa{\A} : \Lip(a) < \infty\}$ is a dense unital subspace  of $\sa{\A}$ and $\{ a \in \sa{\A} : \Lip(a) = 0 \} = \R\unit_\A$.  The following are equivalent:
\begin{enumerate}
\item $(\A,\Lip)$ is a  compact quantum metric space; 
\item the metric $\Kantorovich{\Lip}$ is bounded and there exists $r \in \R, r >0$ such that the set:
\begin{equation*}
\{a \in \dom{\Lip} : \Lip(a) \leq 1 \text{ and } \Vert a \Vert_\A \leq r \}
\end{equation*}
is totally bounded in $\A$ for $\Vert \cdot \Vert_\A $;
\item the set:
\begin{equation*} \{ a+\R1_\A \in \sa{\A}/\R1_\A : a \in \dom{\Lip}, \Lip(a) \leq 1 \}
\end{equation*} is totally bounded in $\sa{\A}/\R1_\A$ for $\Vert \cdot \Vert_{\sa{\A}/\R1_\A}$; 
\item there exists a state $\mu \in \StateSpace (\A)$ such that the set: 
\begin{equation*}
\{ a\in \dom{\Lip} : \Lip (a) \leq 1 \text{ and } \mu(a) = 0 \}
\end{equation*}
is totally bounded in $\A$ for $\Vert \cdot \Vert_\A $;
\item for all $\mu \in \StateSpace (\A)$ the set: 
\begin{equation*}
\{ a\in \dom{\Lip} : \Lip (a) \leq 1 \text{ and } \mu(a) = 0 \}
\end{equation*}
is totally bounded in $\A$ for $\Vert \cdot \Vert_\A $.
\end{enumerate}
\end{theorem}

\Latremoliere's quantum Gromov-Hausdorff propinquity \cite{Latremoliere13, Latremoliere13b, Latremoliere15} is a distance on the class of compact quantum metric spaces and serves as  noncommutative analogue to the Gromov-Hausdorff distance that has proven to be an especially profitable contribution to noncommutative metric geometry by expanding the possiblilities of continuous families of C*-algebras as well as extending the notion of finite-dimensional approximations  \cite{Latremoliere13c, Rieffel15} .  However, in order for propinquity to capture the C*-algebraic structure and have distance $0$ in propinquity to produce a *-isomorphism between the underlying C*-algebras, \Latremoliere\ had the insight  to remedy this by requiring the Lip-norms to have a multiplicative property like that of the Leibniz rule defined in the following Definition (\ref{quasi-Monge-Kantorovich-def}). In particular, the quantum Gromov-Hausdroff propinquity produces a distance on the class of compact quantum metric spaces of Definition (\ref{quasi-Monge-Kantorovich-def}) with this desirable distance $0$ property given, in part, by a *-isomorphism (see Theorem-Definition (\ref{def-thm},(5))).  Thus,  all Lip-norms in this paper are shown to satisfy the following definition.

\begin{definition}[{\cite{Latremoliere15}}]\label{quasi-Monge-Kantorovich-def}
A $(C,D)$-quasi-Leibniz compact quantum metric space $(\A,\Lip)$, for some $C\geq 1$ and $D\geq 0$, is  compact quantum metric space such that
the seminorm $\Lip$ is a \emph{$(C,D)$-quasi-Leibniz Lip-norm}, i.e. for all $a,b \in \dom{\Lip}$:
\begin{equation*}
\max\left\{ \Lip\left(a \circ b \right), \Lip\left(\{a,b\}\right) \right\} \leq C\left(\|a\|_\A \Lip(b) + \|b\|_\A\Lip(a)\right) + D \Lip(a)\Lip(b)\text{,}
\end{equation*}
where $a\circ b = \frac{ab+ba}{2}$ is the Jordan product and $\{a,b\}=\frac{ab-ba}{2i}$ is the Lie product.

When $C=1, D=0$, we call $\Lip$ a {\em Leibniz} Lip-norm.  When we do not specify $C$ and $D$, we call $(\A,\Lip)$ a {\em quasi-Leibniz  compact quantum metric space}.
\end{definition}

The following serves as a summary of results we will use in this paper involving the \Latremoliere's propinquity.
\begin{theorem-definition}[\cite{Latremoliere13, Latremoliere15}]\label{def-thm}
 Let $\mathrm{qLCQMS}$ be the class of all quasi-Leibniz compact quantum metric spaces. There exists a class function $\qpropinquity{}$ from $\mathrm{qLCQMS}\times \mathrm{qLCQMS}$ to $[0,\infty) \subseteq \R$ such that:
\begin{enumerate}
\item for any $(\A,\Lip_\A), (\B,\Lip_\B) \in \mathrm{qLCQMS}$ we have:
\begin{equation*}
 \qpropinquity{}((\A,\Lip_\A),(\B,\Lip_\B)) \leq \max\left\{\diam{\StateSpace(\A)}{\Kantorovich{\Lip_\A}}, \diam{\StateSpace(\B)}{\Kantorovich{\Lip_\B}}\right\}\text{,}
\end{equation*}
\item for any $(\A,\Lip_\A), (\B,\Lip_\B) \in \mathrm{qLCQMS}$ we have:
\begin{equation*}
0\leq \qpropinquity{}((\A,\Lip_\A),(\B,\Lip_\B)) = \qpropinquity{}((\B,\Lip_\B),(\A,\Lip_\A))
\end{equation*}
\item for any $(\A,\Lip_\A), (\B,\Lip_\B), (\alg{C},\Lip_{\alg{C}}) \in \mathrm{qLCQMS}$ we have:
\begin{equation*}
\qpropinquity{}((\A,\Lip_\A),(\alg{C},\Lip_{\alg{C}})) \leq \qpropinquity{}((\A,\Lip_\A),(\B,\Lip_\B)) + \qpropinquity{}((\B,\Lip_\B),(\alg{C},\Lip_{\alg{C}}))\text{,}
\end{equation*}
\item for all  for any $(\A,\Lip_\A), (\B,\Lip_\B) \in \mathrm{qLCQMS}$ and for any bridge $\gamma$ from $\A$ to $\B$ defined in \cite[Definition 3.6]{Latremoliere13}, we have:
\begin{equation*}
\qpropinquity{}((\A,\Lip_\A), (\B,\Lip_\B)) \leq \bridgelength{\gamma}{\Lip_\A,\Lip_\B}\text{,}
\end{equation*}
where $\bridgelength{\gamma}{\Lip_\A,\Lip_\B}$ is defined in \cite[Definition 3.17]{Latremoliere13},
\item for any $(\A,\Lip_\A), (\B,\Lip_\B) \in \mathrm{qLCQMS}$, we have:
\begin{equation*}
\qpropinquity{}((\A,\Lip_\A),(\B,\Lip_\B)) = 0
\end{equation*}
if and only if $(\A,\Lip_\A)$ and $(\B,\Lip_\B)$ are {\em quantum isometric}, i.e. if and only if there exists a *-isomorphism $\pi : \A \rightarrow\B$ with $\Lip_\B\circ\pi = \Lip_\A$, and as quantum isometry is an equivalence relation, we have that $\qpropinquity{}$ induces a metric on the class of equivalence classes up to quantum isometry of quasi-Leibniz compact quantum metric spaces, and
\item if $\Xi$ is a class function from $\mathrm{qLCQMS}\times \mathrm{qLCQMS}$ to $[0,\infty)$ which satisfies Properties {\em(2), (3)} and {\em(4)} above, then:
\begin{equation*}
\Xi((\A,\Lip_\A), (\B,\Lip_\B)) \leq \qpropinquity{}((\A,\Lip_\A),(\B,\Lip_\B))
\end{equation*}
 for all $(\A,\Lip_\A)$ and $(\B,\Lip_\B)$ in $\mathrm{qLCQMS}.$
\end{enumerate}
\end{theorem-definition}

Due to this Theorem-Definition, we may introduce the following convention.
\begin{convention}\label{equiv-class-convention}
Let $\mathrm{CMS}$ denote the class of all compact metric spaces. Let $\mathcal{A}$ be a subclass of $\mathrm{CMS}$, then  by $(\mathcal{A}, \mathrm{GH})$, we mean the class of all equivalence classes up to isometry of compact metric spaces topologized by the quotient topology induced by the Gromov-Hausdorff distance, $\mathrm{GH}$ \cite[Section 7.3]{burago01}, for which $\mathrm{GH}$ induces a metric on this quotient space.  And, when we let $(X,\mathsf{d}_X) \in (\mathcal{A}, \mathrm{GH}),$ we implicitly mean the  equivalence class of $(X,\mathsf{d}_X)$ with repsect to isometry. 

By $(\mathrm{qLCQMS}, \qpropinquity{})$, we mean the class of all equivalence classes up to quantum isometry of Theorem-Definition (\ref{def-thm}) topologized by the quotient topology induced the the quantum Gromov-Hausdorff propinquity, $\qpropinquity{}$.  And, when we take $(\A,\Lip) \in (\mathrm{qLCQMS}, \qpropinquity{}),$ we implicitly mean the equivalence class of $(\A,\Lip) $ with respect to quantum isometry. 
\end{convention}

We also have the following theorem that establishes the quantum Gromov-Hausdorff propinquity as a noncommutative analogue of the Gromov-Hausdorff distance on compact metric space.

\begin{theorem}[{\cite[Theorem 6.6, Corollary 6.4]{Latremoliere13} and \cite[Theorem 13.6]{Rieffel00}}]\label{comm-propinquity-thm}
If given compact metric spaces $(X, \mathsf{d}_X), (Y, \mathsf{d}_Y)$ and we let $\Lip_{\mathsf{d}_X}, \Lip_{\mathsf{d}_Y}$ denote their respective Lipschitz seminorms, then:
 \begin{equation*}
 \qpropinquity{} \left(\left(C(X), \Lip_{\mathsf{d}_X}\right), \left( C(Y), \Lip_{\mathsf{d}_Y}\right)\right) \leq \mathrm{GH}((X, \mathsf{d}_X), (Y, \mathsf{d}_Y)),
 \end{equation*}
 where $\mathrm{GH}$ is the Gromov-Hausdorff distance \cite[Section 7.3]{burago01}.
 
 Moreover, using Convention (\ref{equiv-class-convention}), the class map:
 \begin{equation*}
 \Gamma: (X, \mathsf{d}_X) \in (\mathrm{CMS}, \mathrm{GH}) \mapsto  \left(C(X), \Lip_{\mathsf{d}_X}\right) \in \left(\mathrm{qLCQMS}, \qpropinquity{}\right)
 \end{equation*}
 is a homeomorphism onto its image.
\end{theorem}
A main goal of this paper is to generalize the continuity of $\Gamma$ in this theorem to matrix-valued continuous functions.

\section{Quantum metrics on standard homogeneous C*-algebras}\label{cqms-section}

Given a compact metric space $\left(X, \mathsf{d}_X \right)$ and a finite-dimensional C*-algebra $\A$, the task of equipping $C(X, \A)$ with a Lip-norm may at first seem obvious since we could define the quantity:
\begin{equation*}
L(a)=\sup_{x , y \in X, x \neq y} \frac{\Vert a(x) - a(y) \Vert_\A}{\mathsf{d}_X(x,y)} \text{ for  all $a \in C(X,\A)$}.
\end{equation*}
 However, an immediate issue with this quantity is that the kernel of $L$ is  not $\C1_{C(X,\A)}$ if $\dim(\A)>1.$  Thus, we would immediately not  satisfy the definition of the Lip-norm and Theorem (\ref{Rieffel-thm}) (the theorem that , in part, motivates the term "compact quantum metric space") would be unavailable to us.  Hence, in this section, we present various remedies to this deficit by coupling the quantity $L$ with other quantities and choosing different norms on $\A$ for the quantity in the numerator of $L$ to make a Lip-norm. We will discuss the advantages of each construction and, in the process, provide new Lip-norms on $C(X)$ itself. 
 
 First, we explicitly define what we mean by a standard homogeneous C*-algebra with a remark afterward explaining this definition.
 \begin{definition}\label{standard-h-def}
 A unital separable C*-algebra $\B$ is a {\em standard homogeneous C*-algebra} if there exists a compact metric space $(X, \mathsf{d}_X)$  and a finite-dimensional C*-algebra $\A$ such that $\B=C(X,\A)$, which is the C*-algebra of continuous $\A$-valued functions on $X$ with point-wise algebraic and adjoint operations induced by $\A$,  supremum norm, and the unit is the constant $1_\A$ function on $X$.  
 \end{definition}
 \begin{remark}
 The reason we assume $X$ is compact metric in the previous definition is because the C*-algebra of a compact quantum metric space is always unital by definition and separable by \cite[Proposition 2.11]{Latremoliere13b}. Indeed, if $X$ is compact Hausdorff and $\A$ is finite-dimensional and $C(X,\A)$ is separable, then its state space is compact by unital and metrizable as the unit ball of a the dual of a separable Banach space is metrizable in the weak* topology.  However, $X$ embeds homeomorphically into the state space of $C(X,\A)$ (see proof of Proposition (\ref{any-state-lip-prop})) , which induces a metric on $X$ that agrees with its topology. The terminology "standard" is taken from \cite[Section IV.1.4]{Blackadar-op}.
 \end{remark}
 Now, we define  the family of seminorms we consider throughout this paper, and we note that we define them in the more general setting than Definition (\ref{standard-h-def}), in which the C*-algebra $\A$ in $C(X,\A)$ need not be finite-dimensional.  And, in fact, we can and do prove many interesting properties about the seminorms of the following definition without the assumption that $\A$ is finite-dimensional, and we only assume $\A$ is finite-dimensional when it is necessary in Theorem (\ref{fd-thm}), in that this theorem, in part, provides an equivalence for finite-dimensionality of $\A$. This thus shows that our seminorms are natural choices for Lip-norms on $C(X,\A)$ in the case when $\A$ is finite-dimensional.
\begin{definition}\label{homog-lipschitz-def}
Let $(X, \mathsf{d}_X)$ be a compact metric space and let $\A$ be a unital C*-algebra.  Let $C(X, \A)$ denote the unital  C*-algebra of continuous $\A$-valued functions on $X$ with point-wise algebraic and adjoint operations induced by $\A$, supremum norm on $X$, and the unit $1_{C(X,\A)}$ is the constant $1_\A$ function on $X$.  Let $C(X,\C1_\A)=\{a \in C(X, \A) : a(x) \in \C1_\A \text{ for all } x \in X\}$. Define:
\begin{equation*}
l_{\mathsf{d}_X}^{\mathsf{(n)}} (a)= \sup \left\{ \frac{\Vert a(x) - a(y) \Vert_{\mathsf{n}}}{\mathsf{d}_X (x,y)} : x,y\in X,x \neq y \right\} \text{ for all } a \in C(X, \A)
\end{equation*}
where $\Vert \cdot \Vert_\mathsf{n}$ denotes any norm over $\R$ or $\C$ on $\A$.

Let $\Lip_{\mathsf{d}_X}^{(\mathsf{n}),q}$ denote the following:
\begin{enumerate}
\item if $q=C(X)$, then for all $a \in C(X, \A)$ let:

 $\Lip_{\mathsf{d}_X}^{(\mathsf{n}),q}(a) = \max \left\{ l_{\mathsf{d}_X}^{(\mathsf{n})} (a), \left\Vert a+C(X, \C1_\A) \right\Vert_{C(X,\A)/C(X, \C1_\A) } \right\}$ ;
 
\item if $q=\C$, then for all $a \in C(X, \A)$ let:

 $\Lip_{\mathsf{d}_X}^{(\mathsf{n}),q}(a) = \max \left\{ l_{\mathsf{d}_X}^{(\mathsf{n})} (a), \left\Vert a+\C1_{C(X,\A)} \right\Vert_{C(X,\A)/\C1_{(C(X,\A)}} \right\}$;
 
\item if $\mu \in \StateSpace(C(X, \A))$ is any state and $q=\mu$, then for all $a \in C(X, \A)$ let:

 $\Lip_{\mathsf{d}_X}^{(\mathsf{n}),q}(a) = \max \left\{ l_{\mathsf{d}_X}^{(\mathsf{n})} (a), \left\Vert a-\mu(a)1_{C(X, \A)} \right\Vert_{C(X,\A)} \right\}$.
\end{enumerate}
If $\A=\C$, $\mathsf{n}=\C$ is the usual norm on $\C$, and $q=C(X)$, then we note that  $\Lip_{\mathsf{d}_X}^{(\mathsf{n}),q}=l^{(\mathsf{n})}_{\mathsf{d}_X}$ and denote this by $\Lip_{\mathsf{d}_X}$.

As a convention, when $\mathsf{n}$ or $q$ are not specified, then we implicity assume that they satisfy any of the conditions above.
\end{definition}
We note that the constructions of the above definition are related to  the construction of a norm in \cite[Proposition 4.4]{Latremoliere05b}, where the non-unital case is considered, in which $X$ is a locally compact separable metric space.  However,  the seminorm used there is a norm and uses the norm on $C_0(X, \A)$ in place of $q$ above.   Hence, the norm of \cite[Proposition 4.4]{Latremoliere05b} applied in our setting would vanish only at $0$, which would not provide a possibility for a Lip-norm.  Thus, the fact that we are left to rely on the above choices of $q$ does require us to do more work to prove that the seminorms of Definition (\ref{homog-lipschitz-def}) form Lip-norms if and only if $\A$ is finite-dimensional.

When $\A$ is a finite dimensional C*-algebra, there are many standard norms that can be placed on $\A$, which are automatically equivalent by finite-dimensionality.  Later in Section (\ref{mk-metric-section}), we will focus on one particular norm aside from the C*-norm, the max norm (see Lemma (\ref{norm-equiv-lemma}) and Remark (\ref{max-norm-remark}), the proof of Proposition (\ref{any-state-lip-prop}), and Corollary (\ref{main-cor}), for instances when the max norm is used or mentioned). But, for  now, let's focus on the algebraic properties of the domain of the seminorms of Definition (\ref{homog-lipschitz-def}), and note the following proposition does not assume finite-dimensionality of $\A$ in $C(X,\A)$ and this finite-dimensionality assumption on $\A$ does not appear until Theorem (\ref{fd-thm}), where it is, in fact, a necessity.
\begin{proposition}\label{Leibniz-homog-norm-prop}
Let $(X, \mathsf{d}_X)$ be a compact metric space and let $\A$ be a unital C*-algebra  and let $\mu \in \StateSpace(C(X,\A))$ be a state.

 Using notation from Defintion (\ref{homog-lipschitz-def}), if $\Vert \cdot \Vert_\mathsf{n}$ is a norm on $\A$ over $\R$ or $\C$ that is equivalent to the C*-norm $\Vert \cdot \Vert_{\A}$, then $\ker  \Lip_{\mathsf{d}_X}^{(\mathsf{n}),q}=\C1_{C(X,\A)}$  and  $\dom{\Lip_{\mathsf{d}_X}^{(\mathsf{n}),q}}_{C(X, \A)}$ is a  unital *-subalgebra of $C(X, \A).$ 
 
 Furthermore, if $M>0, N>0$ such that $M \Vert \cdot \Vert_\mathsf{n} \leq \Vert \cdot \Vert_\A \leq N\Vert \cdot \Vert_\mathsf{n}$, then:
\begin{enumerate}
\item if $q$ is either $C(X)$ or  $\C$, then  $\Lip_{\mathsf{d}_X}^{(\mathsf{n}),q}$ is a $\left(N/M,0\right)$-quasi-Leibniz seminorm;
\item if $q=\mu$, then $\Lip_{\mathsf{d}_X}^{(\mathsf{n}),q}$ is a $\left(\max \{N/M, 2\},0\right)$-quasi-Leibniz seminorm.
\end{enumerate}
\end{proposition}
\begin{proof}
For $\ker  \Lip_{\mathsf{d}_X}^{(\mathsf{n}),q}=\C1_{C(X,\A)}$, first note that if $a \in \C1_{C(X, \A)} \subseteq C(X,\C1_\A)$, then since $a$ is constant, $l^{(\mathsf{n})}_{\mathsf{d}_X}(a)=0$ and, for any choice of $q$, the second expression in the definition of $\Lip^{(\mathsf{n}), q}_{\mathsf{d}_X}$ is also $0$ and so $\Lip^{(\mathsf{n}), q}_{\mathsf{d}_X}(a)=0.$  Hence, we have  $\ker  \Lip_{\mathsf{d}_X}^{(\mathsf{n}),q} \supseteq \C1_{C(X,\A)}$.  

Next, let $a \in \ker  \Lip_{\mathsf{d}_X}^{(\mathsf{n}),q}$.  First, consider the case when $q=\C$ or $q=\mu$.  Since $ \Lip_{\mathsf{d}_X}^{(\mathsf{n}),q}(a)=0$, then either: 
\begin{equation*}\left\Vert a+\C1_{C(X, \A)} \right\Vert_{C(X, \A)/\C1_{C(X, \A)}}=0 \text{ or }\Vert a-\mu(a)1_{C(X, \A)}\Vert_{C(X, \A)}=0.\end{equation*}  In either case, we have that $a \in \C1_{C(X, \A)}$ and so $\ker  \Lip_{\mathsf{d}_X}^{(\mathsf{n}),q}=\C1_{C(X,\A)}$.  Second, assume that $q=C(X)$.  If $ \Lip_{\mathsf{d}_X}^{(\mathsf{n}),q}(a)=0$, then $l^{(\mathsf{n})}_{\mathsf{d}_X}(a)=0$ implies that $a$ is constant.  However, the expression $\Vert a+C(X, \C1_\A) \Vert_{C(X, \A)/C(X, \C1_\A)}=0$ implies that $a(x) \in \C1_\A$ for all $x \in X$.  Thus $a$ is a constant scalar and so $a \in \C1_{C(X, \A)}$, which implies that $\ker  \Lip_{\mathsf{d}_X}^{(\mathsf{n}),q}=\C1_{C(X,\A)}$.

Next, consider $\dom{\Lip_{\mathsf{d}_X}^{(\mathsf{n}),q}}_{C(X, \A)}$.  It is clear that $\Lip_{\mathsf{d}_X}^{(\mathsf{n}),q}$ is a seminorm and thus $\dom{\Lip_{\mathsf{d}_X}^{(\mathsf{n}),q}}_{C(X, \A)}$ is a subspace of $C(X, \A)$, and we have already shown that it is unital.  

For subalgebra,  since $\Vert \cdot \Vert_\mathsf{n}$ is equivalent to $\Vert \cdot \Vert_\A$, we have that there exist $M, N >0 $ such that $M \Vert \cdot \Vert_\mathsf{n} \leq \Vert \cdot \Vert_\A \leq N\Vert \cdot \Vert_\mathsf{n}$. Let $a,b \in \dom{\Lip_{\mathsf{d}_X}^{(\mathsf{n}),q}}_{C(X, \A)}$. We then have:
\begin{equation}\label{quasi-Leibniz-eq}
\begin{split}
& \Vert ab(x)-ab(y) \Vert_{\mathsf{n}}\\
 & \leq \frac{1}{M} \Vert ab(x) -ab(y) \Vert_{\A} \\
& \leq \frac{1}{M} \left(\Vert ab(x) -a(x)b(y) \Vert_\A + \Vert a(x)b(y) - a(y)b(y) \Vert_\A \right) \\
&  \leq \frac{1}{M} \left(\Vert a(x) \Vert_\A \cdot \Vert b(x) - b(y) \Vert_\A + \Vert a(x)-a(y) \Vert_\A \cdot \Vert b(y) \Vert_\A \right) \\
& \leq \frac{1}{M} \left(\Vert a \Vert_{C(X,\A)} \cdot \Vert b(x) - b(y) \Vert_\A + \Vert a(x)-a(y) \Vert_\A \cdot \Vert b \Vert_{C(X, \A)}\right)\\
& \leq  \frac{1}{M} \Big( \Vert a \Vert_{C(X,\A)} \cdot N  \Vert b(x)-b(y) \Vert_\mathsf{n}+  N\Vert a(x)-a(y) \Vert_{\mathsf{n}} \cdot \Vert b \Vert_{C(X, \A)}\Big) \\
& = \frac{N}{M} \left( \Vert a \Vert_{C(X,\A)} \cdot \Vert b(x)-b(y) \Vert_\mathsf{n} +  \Vert a(x)-a(y) \Vert_{\mathsf{n}} \cdot \Vert b \Vert_{C(X, \A)}\right) 
\end{split}
\end{equation}
for all $x,y \in X$.  Thus, since $l^{(\mathsf{n})}_{\mathsf{d}_X}(a), l^{(\mathsf{n})}_{\mathsf{d}_X}(b) < \infty$, we have that $l^{(\mathsf{n})}_{\mathsf{d}_X}(ab)< \infty$.  Furthermore, the seminorm $l^{(\mathsf{n})}_{\mathsf{d}_X}$ is $(N/M, 0)$-quasi-Leibniz.  And, clearly, the expression determined by $q$ is finite on all of $C(X, \A)$.  Hence, the unital subspace $ \dom{\Lip_{\mathsf{d}_X}^{(\mathsf{n}),q}}_{C(X, \A)}$ is a subalgebra.  For *-subalgebra, a similar argument to  that of Expression (\ref{quasi-Leibniz-eq}) applies utilizing the equivalence of norms and the fact that the adjoint is an isometry with respect to the C*-norm.  

Finally, note that both quotient norms $\Vert \cdot \Vert_{C(X, \A)/\C1_{C(X, \A)}}$ and $\Vert \cdot \Vert_{C(X, \A)/C(X, \C1_\A)}$ are Leibniz by \cite[Theorem 3.1]{Rieffel11} since both $\C1_{C(X, \A)}$ and $C(X, \C1_\A)$ are unital C*-subalgebras of $C(X, \A)$.  Also, note that for equivalent norms $M \leq N$ or $(N/M)\geq 1$ lest we reach a contradiction.  This along with Expression (\ref{quasi-Leibniz-eq}) provides statement (1) of this proposition.  

For statement (2), we note that the function $a \in C(X, \A) \longmapsto \mu(a)1_{C(X,\A)}$ is a conditional expectation onto $\C1_{C(X, \A)}$.  Hence, by \cite[Lemma 3.2]{Aguilar-Latremoliere15}, we conclude that the seminorm:
\begin{equation*} a \in C(X, \A) \longmapsto \Vert a- \mu(a)1_{C(X, \A)} \Vert_{C(X, \A)}
\end{equation*} is $(2,0)$-quasi-Leibniz, and this establishes statement (2) with Expression (\ref{quasi-Leibniz-eq}).
\end{proof}

We now move onto the analytic properties of the seminorm $\Lip_{\mathsf{d}_X}^{(\mathsf{n}),q}$ on $C(X,\A)$ such as lower semi-continuity and density of the domain. Towards this goal, in Lemma (\ref{classical-lip-dense-lemma}), we prove these properties in the classical case $C(X)$ equipped with its standard Lipschitz seminorm $\Lip_{\mathsf{d}_X}$, which is a well-known result but useful for the proof of the  lemma that follows.    

\begin{lemma}\label{classical-lip-dense-lemma}
Consider the C*-algebra $C(X)$.  The Lipschitz seminorm $\Lip_{\mathsf{d}_X}$ is lower semi-continuous on $C(X)$ with respect to $\Vert \cdot \Vert_{C(X)}$ and its domain $\dom{\Lip_{\mathsf{d}_X}}_{C(X)}$ is dense.
\end{lemma}
\begin{proof}
First, we check lower semi-continuity of $\Lip_{\mathsf{d}_X}$.   Fix $x,y \in X$.  Note that the map $\Lip_{x,y}: f \in C(X) \longmapsto \frac{\vert f(x)-f(y)\vert}{\mathsf{d}_X(x,y)} \in \R$ is continuous.  But, we have that $\Lip_{\mathsf{d}_X}(f)=\sup \left\{ \Lip_{x,y}(f) : x,y \in X \right\}$.  Hence, since a supremum of real-valued lower semi-continuous functions is lower semi-continuous, we have that $\Lip_{\mathsf{d}_X}$ is lower semi-continuous.

Next, we prove density of $\dom{\Lip_{\mathsf{d}_X}}_{C(X)}$ in $C(X)$. By Proposition (\ref{Leibniz-homog-norm-prop}), we have that   $\dom{\Lip_{\mathsf{d}_X}}_{C(X)}$ is a unital *-subalgebra of $C(X)$.  Now, fix $a,b \in X, a \neq b$ and consider the function on $X$ defined by $a_\mathsf{d}(x)=\mathsf{d}_X(a,x)$ for all $x \in X$.  Clearly, the function $a_\mathsf{d} \in C(X)$.  Also, we have for $x,y \in X$ that $\vert a_\mathsf{d}(x)-a_\mathsf{d}(y)\vert=\vert \mathsf{d}_X (a,x) - \mathsf{d}_X(a,y)\vert \leq \mathsf{d}_X(x,y)$.  Hence, the function $a_\mathsf{d} \in \dom{\Lip_{\mathsf{d}_X}}_{C(X)}.$ Finally, $a_\mathsf{d}(b)>0=a_\mathsf{d}(a)$, which implies that $ \dom{\Lip_{\mathsf{d}_X}}_{C(X)}$ separates the points of $X$.  Therefore, the proof is complete by \cite[Stone-Weierstrass Theorem 44.5]{Willard}. 
\end{proof}

Now, we are prepared to generalize the results of Lemma (\ref{classical-lip-dense-lemma}) when $\C$ is replaced by any unital C*-algebra. We  utilize the argument outlined in \cite[Theorem 3.4]{Kaplansky51} to obtain the following.

\begin{lemma}\label{dom-dense-lemma}
Let $(X,\mathsf{d}_X)$ be a compact metric space and let $\A$ be a unital C*-algebra.  Using notation from Defintion (\ref{homog-lipschitz-def}), if $\Vert \cdot \Vert_\mathsf{n}$ is a norm on $\A$ over $\R$ or $\C$ that is equivalent to the C*-norm $\Vert \cdot \Vert_{\A}$, then the seminorm $\Lip_{\mathsf{d}_X}^{(\mathsf{n}),q}$ is lower semi-continuous with respect to $\Vert \cdot \Vert_{C(X, \A)}$ and the sets $\dom{\Lip_{\mathsf{d}_X}^{(\mathsf{n}),q}}_{C(X, \A)}$ and $\dom{\Lip_{\mathsf{d}_X}^{(\mathsf{n}),q}}$ are dense in $C(X, \A)$ and $\sa{C(X, \A)}$, respectively.
\end{lemma}

\begin{proof}
Semi-continuity follows as in the proof of Lemma (\ref{classical-lip-dense-lemma}) along with the fact that $\| \cdot \|_\mathsf{n}$ is equivalent to $\|\cdot \|_\A$.

For density of $\dom{\Lip_{\mathsf{d}_X}^{(\mathsf{n}),q}}_{C(X, \A)}$ in $C(X, \A)$, let $f\in C(X,\A)$.  Let $\varepsilon>0$. As $X$ is compact, $f$ is uniformly continuous, and thus there exists $\delta>0$ such that: 
\begin{equation}\label{deltaepsilon}
\mathsf{d}_X(x,y)<\delta\qquad\implies\qquad\|f(x)-f(y)\|_\A <\varepsilon/2
\end{equation}
for all $x,y \in X$.  Define $U(y,\delta/2)=\{x \in X: \mathsf{d}_X (x,y)<\delta/2\}$ for all $y \in X$.  Again, as $X$ is compact, the open cover $\{ U(y,\delta/2)\subseteq X: y \in X\}$ of $X$ has a finite subcover of $X$ given by $y_1, \ldots, y_n\in X$ such that $\cup_{k=1}^n U(y_k, \delta/2)=X$.  Since $X$ is compact Hausdorff, there exists a partition of unity with respect to the cover $\{U(y_1, \delta/2), \ldots, U(y_n, \delta/2)\}$ by \cite[Proposition IX.4.3.3]{Bourbaki-top2}.  In particular, for each $k \in \{1, \ldots, n\}$, there exists a continuous function $p_k: X \rightarrow [0,1]$ such that $\{x \in X: p_k(x)>0\}\neq \emptyset$ and if we define $V_k=\{x \in X: p_k(x)>0\}$, then   $\{V_1, \ldots, V_n\}$ is an open cover of $X$ and $V_k \subseteq \overline{V_k}^{\mathsf{d}_X} \subseteq U(y_k,\delta/2)$ for each $k \in \{1, \ldots, n\}$. Futhermore, we have $\sum_{k=1}^n p_k=1_{C(X)}$, which is the constant $1$ function on $X$.

 Now, note that by definition of $l^{(\mathsf{n})}_{\mathsf{d}_X}$, we have that if $g \in C(X)$ and $\Lip_{\mathsf{d}_X}(g)< \infty$, then for any $a \in \A$, we have that $a\cdot g \in C(X, \A)$, where $a\cdot g:x \in X \mapsto a\cdot g(x) \in \A$, and $l^{(\mathsf{n})}_{\mathsf{d}_X}(a\cdot g)= \|a\|_{\mathsf{n}}\cdot \Lip_{\mathsf{d}_X}(g) <\infty$.  Next, for all $k\in \{1, \ldots, n\}$, fix some $x_k \in V_k$.  Let $k\in \{1, \ldots, n\}$. If $f(x_k)=0$, then clearly $f(x_k)p_k \in \dom{\Lip_{\mathsf{d}_X}^{(\mathsf{n}),q}}_{C(X, \A)}$.  If $f(x_k)\neq 0$, then let $q_k \in \dom{\Lip_{\mathsf{d}_X}}_{C(X)}$ such that $\| p_k - q_k \|_{C(X)} < \varepsilon/(2n\cdot \|f(x_k)\|_\A)$ by Lemma (\ref{classical-lip-dense-lemma}). Furthermore, by the comments at the beginning of this paragraph, we have that $f(x_k)\cdot q_k \in \dom{\Lip_{\mathsf{d}_X}^{(\mathsf{n}),q}}_{C(X, \A)}$ and: 
 \begin{equation*}\|f(x_k)\cdot p_k - f(x_k)\cdot q_k \|_{C(X, \A)}= \|f(x_k)\|_\A \cdot \|p_k-q_k\|_{C(X)} < \varepsilon/(2n),
 \end{equation*}
 and thus:
 \begin{equation*}
 \left\| \sum_{k=1}^n f(x_k)\cdot p_k - \sum_{k=1}^n f(x_k)\cdot q_k\right\|_{C(X,\A)} < n \cdot \varepsilon/(2n)=\varepsilon/2.
 \end{equation*}
 Based on these observations, define $f_p =  \sum_{k=1}^n f(x_k)\cdot p_k \in C(X, \A)$ and $f_\varepsilon = \sum_{k=1}^n f(x_k)\cdot q_k \in \dom{\Lip_{\mathsf{d}_X}^{(\mathsf{n}),q}}_{C(X, \A)}$, which is in $\dom{\Lip_{\mathsf{d}_X}^{(\mathsf{n}),q}}_{C(X, \A)}$ since $\Lip_{\mathsf{d}_X}^{(\mathsf{n}),q}$ is a seminorm, and note that $\|f_p-f_\varepsilon\|_{C(X, \A)} < \varepsilon/2$.
 
 Now, let $x \in X$.  Since $\{V_1, \ldots, V_n\}$ is a cover of $X$, we have that $\emptyset \neq \{l \in \{1, \ldots, n\} : x \in V_l\}=\{l \in \{1, \ldots, n\} : p_l(x)>0\}$.  Denote $\{l_1, \ldots, l_m\} =\{l \in \{1, \ldots, n\} : p_l(x)>0\}$, and in particular, we have that $p_l(x)=0$ if $l \in \{1, \ldots, n\}\setminus \{l_1, \ldots, l_m\}$.  Since $x \in V_{l_j} \subseteq U(y_{l_j},\delta/2)$, for all $j \in \{1, \ldots, m\}$, we have that $\mathsf{d}_X(x,x_{l_j})< \delta$ for all $j \in \{1, \ldots, m\}$. Hence by Expression (\ref{deltaepsilon}), we gather: 
 \begin{align*}
 \|f(x)-f_p(x)\|_\A& = \left\| \left(\sum_{k=1}^np_k(x)\right)\cdot f(x) - \sum_{k=1}^nf(x_k)\cdot p_k(x) \right\|_\A\\
 & = \left\| \left(\sum_{k=1}^np_k(x)\cdot f(x)\right) - \sum_{k=1}^nf(x_k)\cdot p_k(x) \right\|_\A\\
 & \leq \sum_{k=1}^n p_k(x) \cdot \|f(x)-f(x_k)\|_\A \\
 & = \sum_{j=1}^m p_{l_j}(x)\cdot \|f(x) - f(x_{l_j})\|_\A < \sum_{j=1}^m p_{l_j}(x) \cdot \varepsilon/2  \leq 1 \cdot  \varepsilon/2
 \end{align*}
 since $\sum_{k=1}^n p_k=1_{C(X)}$. As $x \in X$ was arbitrary, we have $\|f-f_p\|_{C(X, \A)}\leq \varepsilon/2$, which implies that $\|f-f_\varepsilon\|_{C(X,\A)}\leq \|f-f_p\|_{C(X, \A)} + \|f_p-f_\varepsilon\|_{C(X, \A)}< \varepsilon/2+\varepsilon/2=\varepsilon$, where $f_\varepsilon \in \dom{\Lip_{\mathsf{d}_X}^{(\mathsf{n}),q}}_{C(X, \A)},$ which establishes that $\dom{\Lip_{\mathsf{d}_X}^{(\mathsf{n}),q}}_{C(X, \A)}$ is dense in $C(X, \A)$. 
 
 Finally, since $\dom{\Lip_{\mathsf{d}_X}^{(\mathsf{n}),q}}_{C(X,\A)}$ is a *-subalgebra of $C(X,\A)$ by Proposition (\ref{Leibniz-homog-norm-prop}), then $\dom{\Lip_{\mathsf{d}_X}^{(\mathsf{n}),q}}_{C(X,\A)} \cap \sa{C(X, \A)} = \dom{\Lip_{\mathsf{d}_X}^{(\mathsf{n}),q}}$ is dense in $\sa{C(X, \A)}$.  Indeed, let $a \in \sa{C(X, \A)}.$  There exists a sequence $(a_n)_{n \in \N} \subset \dom{\Lip_{\mathsf{d}_X}^{(\mathsf{n}),q}}_{C(X,\A)}$ that converges to $a$.  However, as $\dom{\Lip_{\mathsf{d}_X}^{(\mathsf{n}),q}}_{C(X,\A)}$ is a *-subalgebra, we have that $\left(\frac{a_n+a_n^*}{2}\right)_{n \in \N} \subset \dom{\Lip_{\mathsf{d}_X}^{(\mathsf{n}),q}}_{C(X,\A)} \cap \sa{C(X, \A)}= \dom{\Lip_{\mathsf{d}_X}^{(\mathsf{n}),q}}$ and converges to $\frac{a+a^*}{2}=\frac{a+a}{2}=a$, which completes the proof.
\end{proof}
\begin{remark}
In  Lemma (\ref{dom-dense-lemma}),  the proof of the  density of $\dom{\Lip_{\mathsf{d}_X}^{(\mathsf{n}),q}}_{C(X,\A)}$ did not utilize the C*-algebra structure of $\A$, and this density result would be true if $\A$ was any  normed space.  
\end{remark}

We are on our way to the final steps in proving that we have  compact quantum metric spaces.  There are several ways to approach this.  We will use  \cite[Theorem 1.9]{Rieffel98a} (this is equivalence (2) of Theorem (\ref{Rieffel-thm})), which is Rieffel's first characterization of compact  quantum metric spaces and characterizes when the weak * topology on the state space is metrized by the Monge-Kantorovich metric (the quantum metric) of Definition (\ref{Monge-Kantorovich-def}) by a fascinating application of the Arzela-Ascoli Theorem.  The reason we use  equivalence (2) of Theorem (\ref{Rieffel-thm}) is because it requires us to study the diameter of the Monge-Kantorovich metric, which will provide greater insight into the results of Section (\ref{mk-metric-section}) (see Remark (\ref{diam-remark})). This is Proposition (\ref{mk-metric-diam-lemma}).  However, we first prove two lemmas, the first of which is a likely well-known fact about a characterization of  pure states on $C(X, \A)$. The following lemma can be stated in more generality, but in order to avoid introducing more notation, we present it in the context we require.
\begin{lemma}\label{pure-state-lemma}
Let $(X, \mathsf{d}_X)$ be a compact metric space and  $\A$ be a unital C*-algebra.  Let $x \in X$ and $\phi$ be a pure state on $\A$.  If we define $\phi_x(a)=\phi(a(x))$ for all $a \in C(X, \A)$, then $\phi_x$ is a pure state on $C(X, \A)$.

Furthermore, if $\mu$ is a pure state on $C(X,\A)$, then there exists $x \in X$ and a pure state $\phi$ on $\A$ such that $\mu=\phi_x.$
\end{lemma}
\begin{proof}
It is clear that $\phi_x$ defines a state on $C(X, \A)$ for all $x \in X$ and all pure states $\phi$ on $\A$.  We show that $\phi_x$ is a pure state by verifying the final statement in the lemma. 

Let $\A \odot C(X)$ denote the algebra over $\C$ formed over the algebraic tensor product of $\A$ and $C(X)$ \cite[Section 3.1]{Brown-Ozawa}. Let $\A \otimes C(X)$ denote the C*-algebra formed by the  completion of $ \A \odot C(X)$ with respect to a C*-norm.  This C*-norm exists and is unique since $C(X)$ is commutative, which is why we do not decorate the tensor (see \cite[Proposition 2.4.2 and Proposition 3.6.12]{Brown-Ozawa}). Note that $\A \odot C(X)$  is dense in $\A \otimes C(X)$ by definition.

Let $T: \A \otimes C(X) \rightarrow C(X,\A)$ denote the canonical *-isomorphism for which on elementary tensors is given by $T(a\otimes f) = a\cdot f$, where $a\cdot f: x \in X \mapsto a\cdot f(x) \in \A$ (see \cite[Theorem 6.4.17]{Murphy90}). Let $\mu$ be a pure state on $C(X, \A)$. Then $\mu\circ T$ is a pure state on $\A \otimes C(X)$.  Now, since $C(X)$ is commutative, there exists a pure state  $\phi$ on $\A$  and a pure state $\nu$ on $C(X)$ such that $(\mu \circ T)|_{\A \odot C(X)}=\phi \odot \nu$ by \cite[Corollary 3.4.3]{Brown-Ozawa}, where $\phi \odot \nu$ is a complex-valued linear map on $\A \odot C(X)$ given on elementary tensors by $(\phi \odot \nu)(a \otimes f)=\phi(a)\nu(f)$ (we do not distinguish elementary tensors with $\odot$ since these elements are in $\A \odot C(X) \subseteq \A \otimes C(X)$ by definition). However, as $\nu$ is a pure state on $C(X)$, there exists $x \in X$ such that $\nu=\delta_x$, where $\delta_x: f \in C(X) \mapsto f(x) \in \C$ is the Dirac point mass at $x$ (combine \cite[Thereom VII.8.7]{Conway90} and \cite[Theorem 5.1.6]{Murphy90}). And, thus $(\mu \circ T)|_{\A \odot C(X)}= \phi \odot \delta_x$. Hence:
\begin{align*}
\mu (T(a \otimes f))&=\mu\circ T (a \otimes f)=(\phi \odot \delta_x)(a\otimes f)=\phi(a)\delta_x(f)\\
& = \phi(a)f(x)=\phi(a\cdot f(x))=\phi(T(a\otimes f)(x))=\phi_x(T(a\otimes f))
\end{align*}
for all $a \in \A, f \in C(X)$ since $a\otimes f \in \A \odot C(X)$.
By linearity, it is immediate that $\mu$ and $\phi_x$ agree on $T(\A \odot C(X))$. However, as $T$ is a *-isomorphism, we have that $T(\A \odot C(X))$ is dense in $C(X, \A)$, which implies that $\mu=\phi_x$ on $C(X,\A)$ by continuity of states. This completes the proof.
\end{proof}
The next lemma serves as a tool that reduces the problem of finding upper bounds to the Monge-Kantorovich metric in general and regardless of whether it metrizes the weak* topology, which is valid for more general C*-algebras than $C(X,\A)$.
\begin{lemma}\label{pure-mk-metric-lemma}
Let $\A$ be a unital C*-algebra.  If $\Lip$ is a seminorm defined on $\sa{\A}$ such that $\dom{\Lip}$ is dense in $\sa{\A}$, then if $D \in [0,\infty]$ such that $\Kantorovich{\Lip}(\mu,\nu)\leq D$ for all pure states $\mu,\nu$ on $\A$, then $\Kantorovich{ \Lip}(\mu,\nu)\leq D$ for all states $\mu,\nu \in \StateSpace(\A).$
\end{lemma}
\begin{proof}The conditions on $\Lip$ are given so that $\Kantorovich{\Lip}$ is a metric (possibly taking value $+\infty$) on $\StateSpace(\A)$ (see Remark (\ref{mk-metric-dense-remark})).
Let $A$ be a subset of a vector space $V$.  Denote the Minkowski sum of $A$ with itself by $A+A=\{a+b\in V:a \in A, b \in A \}$ and denote $A-A=\{a-b\in V:a \in A, b \in A \}$ and $-A=\{-a\in V:a \in A\}$.  Let $\mathrm{co}(A)$ denote the convex hull of $A$.  It is a routine exercise to show that $\mathrm{co}(A-A)=\mathrm{co}(A+(-A))=\mathrm{co}(A)+\mathrm{co}(-A)=\mathrm{co}(A)-\mathrm{co}(A)$.

If $D=\infty$, then the conclusion is clear, so assume that $D \in [0,\infty)$.  Now, let $\mathscr{PS}(\A)$ denote the pure states of $\A$.  First, assume that $\mu,\nu \in \StateSpace(\A)$ such that $\mu,\nu\in \mathrm{co}(\mathscr{PS}(\A))$. In particular, we have $\mu-\nu \in \mathrm{co}(\mathscr{PS}(\A))-\mathrm{co}(\mathscr{PS}(\A))=\mathrm{co}(\mathscr{PS}(\A)-\mathscr{PS}(\A))$ by the first paragraph. Hence, there exist $n \in \N$ and  $\phi_0, \ldots, \phi_n,\psi_0, \ldots, \psi_n \in \mathscr{PS}(\A)$ and $r_0, \ldots, r_n \in [0,1]$ with $1=\sum_{j=0}^n r_j $ and $\mu-\nu=\sum_{j=0}^n r_j (\phi_j-\psi_j)$. Fix $a \in \dom{\Lip}, \Lip(a) \leq 1$:
\begin{align*}
|\mu(a) -\nu(a)|& =|(\mu-\nu)(a)| =\left| \sum_{j=0}^n r_j (\phi_j-\psi_j)(a) \right|\\
& =\left| \sum_{j=0}^n r_j (\phi_j(a)-\psi_j(a)) \right| \leq   \sum_{j=0}^n r_j \left| (\phi_j(a)-\psi_j(a)) \right| \leq \sum_{j=0}^n r_j\cdot D = D.
\end{align*}
Now, let $\mu,\nu \in \StateSpace(\A)$.  Since $\A$ is unital, we have that $\mu,\nu \in \overline{\mathrm{co}(\mathscr{PS}(\A))}^{weak^*}$ by \cite[Corollary 5.1.10]{Murphy90}. Thus, there exist nets $(\mu_\lambda)_{\lambda \in \Delta}, (\nu_\gamma)_{\gamma\in \Gamma}\subseteq \mathrm{co}(\mathscr{PS}(\A))$ that converge to $\mu,\nu$, respectively, in the weak* topology.  However, by the above, we have $|\mu_\lambda(a)-\nu_\gamma(a)|\leq D$ for all $\lambda \in \Delta, \gamma \in \Gamma$.  Hence, we gather that $|\mu(a) -\nu(a)|\leq D$ by definiton of the weak* topology.  Since $a \in \dom{\Lip}, \Lip(a) \leq 1$ was arbitrary, we have $\Kantorovich{\Lip}(\mu,\nu) \leq D$, which completes the proof since $\mu,\nu \in \StateSpace(\A)$ were arbitrary.
\end{proof}

We can now combine these results to provide upper bounds for the diameter of the Monge-Kantorovich metric for our seminorms on $C(X,\A)$.
\begin{proposition}\label{mk-metric-diam-lemma}
Let $(X, \mathsf{d}_X)$ be a compact metric space and  $\A$ be a unital C*-algebra. Let $\varphi \in \StateSpace(C(X,\A)).$    

Using notation from Defintion (\ref{homog-lipschitz-def}), if  $\Vert \cdot \Vert_\mathsf{n}$ is any norm on $\A$ over $\R$ or $\C$ equivalent to $\A$ so that there exist $M>0, N>0$ such that $M \Vert \cdot \Vert_{\mathsf{n}} \leq \Vert \cdot \Vert_\A \leq N \Vert \cdot \Vert_\mathsf{n}$,  then if:
\begin{enumerate}
\item
\begin{enumerate}
\item $q=C(X)$, then $\diam{\StateSpace(C(X, \A))}{\Kantorovich{\Lip_{\mathsf{d}_X}^{(\mathsf{n}),q}}} \leq 2+N\cdot \diam{X}{ \mathsf{d}_X} $

\item and if $\A=\C$, then $\diam{\StateSpace(C(X, \A))}{\Kantorovich{\Lip_{\mathsf{d}_X}^{(\mathsf{n}),q}}} \leq N\cdot \diam{X}{ \mathsf{d}_X};$
\end{enumerate}
\item $q=\C$ or $q=\varphi$, then $\diam{\StateSpace(C(X, \A))}{\Kantorovich{\Lip_{\mathsf{d}_X}^{(\mathsf{n}),q}}} \leq 2.$
\end{enumerate}
\end{proposition}
\begin{proof}
We begin by noting that as $\dom{\Lip_{\mathsf{d}_X}^{(\mathsf{n}),q}}$ is dense in $\sa{C(X,\A)}$ by Lemma (\ref{dom-dense-lemma}), we have that $\Kantorovich{\Lip_{\mathsf{d}_X}^{(\mathsf{n}),q}}$ is a metric on $\StateSpace(C(X, \A))$ (possibly taking value $+\infty$) --- see Remark (\ref{mk-metric-dense-remark}). 

For (1)(a) and (1)(b), assume that $q=C(X)$, we will combine both Lemma (\ref{pure-state-lemma}) and Lemma (\ref{pure-mk-metric-lemma}) to achieve the result. First, we establish (1)(a). Thus, let $\mu, \nu$ be pure states on $C(X,\A)$.  By Lemma (\ref{pure-state-lemma}), there exist $x,y\in X$ and pure states $\phi, \psi$ on $\A$ such that $\mu=\phi_x$ and $\nu=\psi_y$.  Fix $a \in \dom{\Lip_{\mathsf{d}_X}^{(\mathsf{n}),q}}, \Lip_{\mathsf{d}_X}^{(\mathsf{n}),q}(a)\leq 1$.  Let $\varepsilon>0$.   Since $q=C(X)$, we have that there exists $b \in C\left(X, \C1_\A\right)$ such that $\|a-b\|_{C(X,A)} \leq 1+\varepsilon$.  Therefore, as $b(x) \in \C1_\A$, we have  $|\phi_x(b)-\psi_x(b)|= |\phi(b(x))-\psi(b(x))|=0$ since $\phi,\psi$ are states on $\A$ and agree on scalars of $\A$.  Therefore:
\begin{equation}\label{quo-lip-eq}
\begin{split}
|\phi_x(a)-\psi_x(a) |&=|\phi_x(a)-\phi_x(b)+\psi_x(b)-\psi_x(a)|\\
& \leq |\phi_x(a-b)| +|\psi_x(a-b)| \leq 2 \cdot \|a-b\|_\A  \leq 2(1+\varepsilon)=2+2\varepsilon.
\end{split}
\end{equation}
Since $\varepsilon>0$ was arbitrary, we have that $|\phi_x(a)-\psi_x(a)|\leq 2$.  Furthermore, since $\Lip_{\mathsf{d}_X}^{(\mathsf{n}),q}(a)\leq 1$, we gather:
\begin{align*}
|\psi_x(a)-\psi_y(a) |& =|\psi(a(x))-\psi(a(y))| = |\psi(a(x)-a(y))|\leq \|a(x)-a(y)\|_\A \\
& \leq N \cdot \|a(x)-a(y)\|_\mathsf{n}  \leq N \cdot \mathsf{d}_X(x,y) \leq N \cdot \diam{X}{\mathsf{d}_X},
\end{align*}
and:
\begin{align*}
|\mu(a)-\nu(a)|& = |\phi_x(a) - \psi_y(a)|\\
& \leq |\phi_x(a) -\psi_x(a)|+|\psi_x(a)-\psi_y(a)| \leq 2+N \cdot \diam{X}{\mathsf{d}_X}.
\end{align*}
Since $a \in \dom{\Lip_{\mathsf{d}_X}^{(\mathsf{n}),q}}, \Lip_{\mathsf{d}_X}^{(\mathsf{n}),q}(a)\leq 1$ was arbitrary, we have that $\Kantorovich{\Lip_{\mathsf{d}_X}^{(\mathsf{n}),q}}(\mu,\nu) \leq 2+N \cdot \diam{X}{\mathsf{d}_X}.$  Thus (1)(a) is proven by Lemma (\ref{pure-state-lemma}) as $\mu,\nu$ were arbitrary pure states on $C(X,\A)$.  For (1)(b), if $\A=\C$, then note that Expression (\ref{quo-lip-eq}) would be $0$ as there is only one state on $\C$.  This then proves (1)(b).

For (2), the case $q=\C$ is immediate from \cite[Proposition 1.6]{Rieffel98a} and that $\|a+\C1_\A \|_{\A/\C1_\A} \leq \Lip(a)$ for all  $a \in \sa{\A}$.

 For the case $q=\varphi$, the fact that:
\begin{align*} \left\|a+\C1_{C(X,\A)}\right\|_{C(X,\A)/\C1_{C(X,\A)}}&=\inf_{b \in \C1_{C(X,\A)}} \|a-b\|_{C(X,\A)}  \leq \left\| a-\varphi(a)1_{C(X,\A)}\right\|_{C(X,\A)}
\end{align*} for all $a \in \sa{\A}$ completes the proof again by the argument of case $q=\C$.
\end{proof}

Now, we are prepared  to prove our main result of this section, which is that the seminorms of Definition (\ref{homog-lipschitz-def}) will be quasi-Leibniz Lip-norms for $C(X,\A)$ if and only if $\A$ is finite-dimensional, and we note that this is the first result that not only assumes finite-dimensionality of $\A$, but also requires finite-dimensionality of $\A$.
\begin{theorem}\label{fd-thm}
Let $(X, \mathsf{d}_X)$ be a compact metric space and let $\A$ be a unital C*-algebra and let $\mu \in \StateSpace(C(X,\A))$ be a state.  

 Using notation from Defintion (\ref{homog-lipschitz-def}),  the following three statements are equivalent:
 \begin{itemize}
 \item[(i)]  the pair $\left(C(X, \A) , \Lip_{\mathsf{d}_X}^{(\mathsf{n}),q} \right)$ is a quasi-Leibniz  compact quantum metric space;
 \item[(ii)]  $\A$ is finite- dimensional; 
 \item[(iii)]  $C(X,\A)$ is a standard homogeneous C*-algebra of Definition (\ref{standard-h-def}).
 \end{itemize}
 
 Furthermore, if $\A$ is finite-dimensional and   $M>0, N>0$ such that $M \cdot \Vert \cdot \Vert_\mathsf{n} \leq \Vert \cdot \Vert_\A \leq N\cdot \Vert \cdot \Vert_\mathsf{n}$, then:
\begin{enumerate}
\item if $q$ is either $C(X)$ or  $\C$, then  $\Lip_{\mathsf{d}_X}^{(\mathsf{n}),q}$ is a $\left(N/M,0\right)$-quasi-Leibniz Lip-norm;
\item if $q=\mu$, then $\Lip_{\mathsf{d}_X}^{(\mathsf{n}),q}$ is a $\left(\max \{N/M, 2\},0\right)$-quasi-Leibniz Lip-norm.
\end{enumerate}
\end{theorem}
\begin{proof}
For (ii)$\implies$(i), assume that $\A$ is finite-dimensional, we begin by showing that the pair $\left(C(X, \A) , \Lip_{\mathsf{d}_X}^{(\mathsf{n}),q} \right)$ is a quasi-Leibniz  compact quantum metric space. By Proposition (\ref{Leibniz-homog-norm-prop}) and  Lemma (\ref{dom-dense-lemma}), all that remains is to show that the Monge-Kantorovich metric metrizes the weak* topology of the state space. 
Also, by Proposition (\ref{mk-metric-diam-lemma}) and  equivalence (2) of Theorem (\ref{Rieffel-thm}), we only need to verify that there exists some $D>0$ such that the set:
\begin{equation*}
\mathcal{B}_{1,D}= \left\{ a \in \sa{C(X, \A)} : \Lip_{\mathsf{d}_X}^{(\mathsf{n}),q}(a) \leq 1\text{ and } \Vert a \Vert_{C(X, \A)}\leq D\right\}
\end{equation*}
is totally bounded. 

Let $D \in (0, \infty)$.  The set $\mathcal{B}_{1,D}$ is equicontinuous since $l_{\mathsf{d}_X}^{(\mathsf{n})} (a) \leq 1$ for all $a\in \mathcal{B}_{1,D}$ and $\Vert \cdot \Vert_\mathsf{n}$ is equivalent to $\Vert \cdot \Vert_\A$ by finite-dimensionality.  Next, fix $x \in X$, the set $\{a(x) \in \A : a \in \mathcal{B}_{1,D}\} \subseteq \{ b \in \A : \Vert b \Vert_\A \leq D \}$.   Since $\A$ is finite dimensional, the set $\{a(x) \in \A : a \in \mathcal{B}_{1,D}\}$ is totally bounded  for each $x \in X$. Therefore, by a generalization of the Arzela-Ascoli Theorem \cite[Theorem 7.47.1]{Munkres} and a characterizaion of the topology on $C(X,\A)$ \cite[Theorem 7.46.7 and 7.46.8]{Munkres}, the set $\mathcal{B}_{1,D}$ is totally bounded in $C(X, \A)$  since $X$ is compact. Thus, by Theorem (\ref{Rieffel-thm}), this direction is complete.

For (i)$\implies$(ii), assume that the pair $\left(C(X, \A) , \Lip_{\mathsf{d}_X}^{(\mathsf{n}),q} \right)$ is a quasi-Leibniz  compact quantum metric space.  By Theorem (\ref{Rieffel-thm}), the set $\mathcal{B}_{1,D}$ is totally bounded for some $D \in (0, \infty)$.  Note that it follows that $\mathcal{B}_{E,F}$ is also totally bounded for all $E,F \in (0,\infty)$ by scaling. Now, the space of constant functions of $C(X, \A)$ denoted by $K(X, \A)$ is canonically *-isomorphic to $\A$.  Denote this *-isomorphism by:
\begin{equation*}\kappa : a \in \A \longmapsto (x \mapsto a)\in K(X, \A).
\end{equation*}  
Assume that $a\in \A$ such that $\Vert a \Vert_\A \leq 1$.  Then, we have that $\Vert\kappa(a) \Vert_{C(X, \A)} \leq 1$ and  $l_{\mathsf{d}_X}^{(\mathsf{n})}(\kappa(a))=0$ since $\kappa (a)$ is constant.  First consider when $q=C(X)$ or $q=\C$, then since quotient norms are bounded above by the norm that they are induced by, we have:
\begin{equation*}\Lip_{\mathsf{d}_X}^{(\mathsf{n}),q} (\kappa(a)) \leq \Vert\kappa(a) \Vert_{C(X, \A)} \leq 1.
\end{equation*}  
Therefore, the element $\kappa(a) \in \mathcal{B}_{1,1}$. In particular, the set:
\begin{equation*}
\kappa\left(\{a\in \A : \Vert a \Vert_\A \leq 1\}\right) \subseteq \mathcal{B}_{1,1},
\end{equation*} and thus $\kappa\left(\{a\in \A : \Vert a \Vert_\A \leq 1\}\right)$ is totally bounded, which implies that  $ \{a\in \A : \Vert a \Vert_\A \leq 1\}$ is totally bounded since $\kappa$ is a *-isomorphism.   Now, if $q=\mu$, then we would have that: 
\begin{equation*}\Lip_{\mathsf{d}_X}^\A (\kappa(a)) =\left\Vert \kappa(a)-\mu(\kappa(a))1_{C(X, \A)} \right\Vert_{C(X,\A)} \leq 2\Vert\kappa(a) \Vert_{C(X, \A)} \leq 2
\end{equation*} 
for $a \in \A$ such that $\Vert a \Vert_\A \leq 1$.  Thus, since $\mathcal{B}_{2,1}$ is also totally bounded, the same argument shows that  the unit ball of $\A$  is totally bounded.
 Thus, in either case, the C*-algebra $\A$ is finite-dimensional by \cite[1.9$(d)$]{Rudin91}.  (ii)$\iff$(iii) is given by definition. Statements (1) and (2) immdeitalely follow by Proposition (\ref{Leibniz-homog-norm-prop}).
\end{proof}
\begin{remark}
Theorem (\ref{fd-thm}) does not imply that there cannot be quanutm metric structure on $C(X, \A)$ for $\A$ infinite dimensional.  It simply implies that the family of seminorms of Definition  (\ref{homog-lipschitz-def}) only provide quantum metric structure on $C(X, \A)$, when $\A$ is finite dimensional, and this, in turn,  strongly suggests that our seminorms are a natural choice for the quantum metric structure on $C(X, \A)$, when $\A$ is finite dimensional.
\end{remark}
The following corollary focuses on a particular case of interest and a motivating idea for this paper.  It shows that we can still recover the classical structure of $\left(C(X), \Lip_{\mathsf{d}_X}\right)$ within $\left(C(X, \A) , \Lip_{\mathsf{d}_X}^{(\mathsf{n}),q} \right)$ for any finite-dimensional C*-algebra $\A$ and a particular choice of $\mathsf{n}$ and $q$.
\begin{corollary}\label{C(X)-cor}
If $(X, \mathsf{d}_X)$ is a compact metric space and $\A$ is a finite-dimensional C*-algebra, then using notation from Defintion (\ref{homog-lipschitz-def}), the pair $\left(C(X, \A),\Lip_{\mathsf{d}_X}^{(\A),C(X)}\right)$ is a Leibniz  compact quantum metric space such that $\Lip_{\mathsf{d}_X}^{(\A),C(X)}$ recovers  the standard Lipschitz seminorm on $C(X)$ denoted by $\Lip_{\mathsf{d}_X}$ from the canonical *-isomorphism of $C(X)$ onto $C(X, \C1_\A)$ given by:
\begin{equation*}
\mathfrak{c}_X : f \in C(X) \longmapsto (x \in X \mapsto f(x)1_\A) \in C(X, \C1_\A).
\end{equation*}
Hence, we have  $\Lip_{\mathsf{d}_X} =\Lip_{\mathsf{d}_X}^{(\A),C(X)} \circ  \mathfrak{c}_X$, and thus 
\begin{equation*}
\qpropinquity{}\left(\left(C(X), \Lip_{\mathsf{d}_X}\right),\left(C(X, \C1_\A), \Lip_{\mathsf{d}_X}^{(\A),C(X)}\right)\right)=0
\end{equation*}
In particular $\left(C(X), \Lip_{\mathsf{d}_X}\right)=\left(C(X, \A), \Lip_{\mathsf{d}_X}^{(\A),C(X)}\right)$, when $\A=\C$.
\end{corollary}
\begin{proof}
Fix $f \in C(X)$, we have:
\begin{align*}
\Lip_{\mathsf{d}_X}^{(\A),C(X)} \left(\alg{c}_X(f)\right)&= \max \left\{ l^{(\A)}_{\mathsf{d}_X}\left(\alg{c}_X(f)\right), \left\Vert \alg{c}_X(f)+C(X, \C1_\A) \right\Vert_{C(X, \A)/C(X, \C1_\A)}\right\}  \\
& =\max \left\{ l^{(\A)}_{\mathsf{d}_X}\left(\alg{c}_X(f)\right), 0\right\}\\
&= \sup \left\{ \frac{\left\Vert \alg{c}_X(f)(x) - \alg{c}_X(f)(y) \right\Vert_\A}{\mathsf{d}_X(x,y)}: x,y\in X, x \neq y \right\}\\
& = \sup \left\{ \frac{\left\Vert f(x)1_\A - f(y) 1_\A \right\Vert_\A}{\mathsf{d}_X(x,y)}: x,y\in X, x \neq y \right\}\\
& = \sup \left\{ \frac{\left\vert f(x) - f(y)\right\vert \cdot \left\Vert 1_\A \right\Vert_\A}{\mathsf{d}_X(x,y)}: x,y\in X, x \neq y \right\}=\Lip_{\mathsf{d}_X}(f).
\end{align*}
This along with Theorem (\ref{fd-thm}) completes the proof.
\end{proof}

The last result of this section shows that when $\A$ is finite dimensional, then any two seminorms  of Definition (\ref{homog-lipschitz-def}) are equivalent regardless of choice of $\mathsf{n}$ or $q$ (this is quite surprising since the quotient norms associated to $q=C(X)$ and $q=\C$ have different kernels on $C(X, \A)$ as long as $\A\not\cong \C$ and thus cannot be equivalent), and the map $\alg{c}_X$ of the above corollary is bi-Lipschitz.   Note that the following  Proposition (\ref{C(X)-prop}) rests mainly on a result of \Latremoliere  \ in \cite{Latremoliere16b}.  
First, a definition:
\begin{definition}\label{bi-lipschitz-def}
Let $(X, \mathsf{d}_X)$ and $(Y, \mathsf{d}_Y)$ be two pseudo-metric spaces, where pseudo means that distance $0$ need no provide equal elements.  

Fix $N >0$. A map $\pi : (X, \mathsf{d}_X) \longrightarrow (Y, \mathsf{d}_Y)$ is {\em $N$-Lipschitz} if $\mathsf{d}_Y(\pi(a),\pi(b)) \leq N\cdot  \mathsf{d}_X(a,b)$ for all $a,b \in X$.

The map $\pi$ is {\em bi-Lipschitz}, if furthermore, there exists $M>0$ such that $M\cdot \mathsf{d}_X (a,b) \leq \mathsf{d}_Y(\pi(a),\pi(b)) \leq N\cdot  \mathsf{d}_X(a,b)$ for all $a,b \in X$.
\end{definition} 
\begin{proposition}\label{C(X)-prop}
Let $(X, \mathsf{d}_X)$ be a compact metric space and $\A$ be a finite-dimensional C*-algebra. Let $\mu \in \StateSpace(C(X))$. 

Using notation from Defintion (\ref{homog-lipschitz-def}), if $\Vert \cdot \Vert_{\mathsf{n}}$ and $\Vert \cdot \Vert_{\mathsf{m}}$ are norms over $\R$ or $\C$ on $\A$ and $q$ and $p$ are either $C(X), \C$ or $\mu$, then the seminorms $\Lip_{\mathsf{d}_X}^{(\mathsf{n}),q}$ and $\Lip_{\mathsf{d}_X}^{(\mathsf{m}),p}$ are equivalent.  

Futhermore, the map $\alg{c}_X$ of Corollary (\ref{C(X)-cor}) is bi-Lipschitz with respect to any $\Lip_{\mathsf{d}_X}^{(\mathsf{n}),q}$.
\end{proposition}
\begin{proof}
Since $\A$ is finite-dimensional, the norms $\Vert \cdot \Vert_{\mathsf{n}}$ and $\Vert \cdot \Vert_{\mathsf{m}}$ are equivalent.  Hence, the domains $\dom{\Lip_{\mathsf{d}_X}^{(\mathsf{n}),q}}=\dom{\Lip_{\mathsf{d}_X}^{(\mathsf{m}),p}}$ since the terms determined by $q$ and $p$ are finite on all of $C(X, \A)$.  Therefore, by Theorem (\ref{fd-thm}) and   \cite[Corollary 2.5]{Latremoliere16b}, the seminorms $\Lip_{\mathsf{d}_X}^{(\mathsf{n}),q}$ and $\Lip_{\mathsf{d}_X}^{(\mathsf{m}),p}$ are equivalent.  The last statement of this proposition follows from this and Corollary (\ref{C(X)-cor}).
\end{proof}

\begin{remark}
Corollary (\ref{C(X)-cor}) and Proposition (\ref{C(X)-prop}) may suggest that the Lip-norm of Corollary (\ref{C(X)-cor}) is the best choice from our family of seminorms for a Lip-norm on $C(X, \A)$ for $X$ compact metric and $\A$ finite-dimensional.  However, we argue that it can depend on the situation.  First we notice the advantage of a different norm on $\A$ than that of the C*-norm $\Vert \cdot \Vert_\A$  throughout this paper, which is the max norm defined in Lemma (\ref{norm-equiv-lemma}) (see Remark (\ref{max-norm-remark}), the proof of Proposition (\ref{any-state-lip-prop}), and Corollary (\ref{main-cor}), for instances when the max norm is used or mentioned).  Thus, allowing for flexibility on the norm of $\A$ seems advantageous.  

In regards to why we consider other $q$'s aside from $q=C(X)$, we first consider $q=\C$.  Let $Y$ be a compact metric space and $\B$ be a finite-dimensional C*-algebra.  Note that any unital *-monomorphism $\pi : C(X, \A) \longmapsto C(Y, \B)$, will be an isometry with respect to the quotient norms associated to $q=\C$ but not necessarily preserve the quotient norms associated to $q=C(X)$. So, if  $q=C(X)$, then it would be much more difficult to verify if $\pi$ were contractive with respect to the Lip-norms of $C(X, \A)$  and $C(Y, \A)$ than in the case of $q=\C$.  And, contractivity is vital to results pertaining to the Gromov-Hausdorff propinquity (see \cite[Proposition 8.5]{Rieffel00} and \cite[Theorem 3.5]{Aguilar-Latremoliere15}, for example).  The reason to consider $q=\mu$ is similar to $q=\C$, except that a unital *-monomorphism need not preserves all states, but it does occur often enough --- especially in the setting of inductive sequences and inductive limits ---, and the quantity given by $q=\mu$ in the Lip-norm can be much easier to calculate than a quotient norm.
\end{remark}

\section{Isometries from metrics into quantum metrics}\label{mk-metric-section}
In order for our seminorms on $C(X,\A)$ of the previous section to be a suitable generalization of the Lipschitz seminorm on $C(X)$, we claim that there should be a natural embedding of $X$ into the state space of $C(X,\A)$ which captures some of the metric structure of $X$ using the Monge-Kantorovich metric and not just the topological structure. This is because this happens in the classical case when $\A=\C$ as dicussed in the introduction.  When $\A$ is finite-dimensional, the purpose of this section to verify that we can accomplish this claim  for our entire family seminorms of Definition (\ref{homog-lipschitz-def}) by providing a bi-Lipschitz embedding that can be strengthened to an isometry in many intuitive cases that are still valid for all finite-dimensional $\A$ and compact metric $X$. 

For the remainder of this paper, we restrict our attention to finite-dimensional C*-algebras that are {\em equal}  to finite direct sums of complex-valued matrix algebras.  This is so that we can provide explicit estimates and the results of this section are still true for any finite-dimensional C*-algebra $\A$ in $C(X,\A)$, which is explained in the following remark.
\begin{remark}\label{fd-remark}
  Since every finite-dimenisonal C*-algebra $\B$ is *-isomorphic to a  finite direct sums of complex-valued matrix algebras \cite[Theorem III.1.1]{Davidson}, we assume that every finite-dimensional C*-algebra is equal to $\oplus_{k=0}^n M_{m_k}(\C)$ for some $n \in \N, m_k \in \N \setminus \{0\}$ for $ k \in \{0, \ldots,n \} $ for the rest of this paper. This will cause no loss of generality for the results of this section.  Indeed, let $(X,\mathsf{d}_X)$ be a compact metric space,  and assume $\pi: \B \longrightarrow \oplus_{k=0}^n M_{m_k}(\C)$ is a *-isomorphism, then the map:
\begin{equation*}
\Pi: b \in  C(X, \B) \longmapsto \left(x \in X \mapsto \pi(b(x))\right) \in C\left(X, \oplus_{k=0}^n M_{m_k}(\C)\right)
\end{equation*}
is a unital *-isomorphism, and $\Lip^{(\mathsf{n}),q}_{\mathsf{d}_X} \circ \Pi$ would define a Lip-norm for  $ C(X, \B) $ with the same properties of $\Lip$, and furthermore, the dual map:
\begin{equation*}\Pi^*: \mu \in \StateSpace\left(C\left(X, \oplus_{k=0}^n M_{m_k}(\C)\right)\right) \mapsto \mu\circ \Pi \in \StateSpace(C(X,\B))
\end{equation*} is an isometry between the associated Monge-Kantorovich metrics of $\Lip^{(\mathsf{n}),q}_{\mathsf{d}_X}$ and $\Lip^{(\mathsf{n}),q}_{\mathsf{d}_X} \circ \Pi$ by \cite[Thoerem 6.2]{Rieffel00}.
\end{remark}

\begin{definition}\label{matrix-units-def}
Let $\A=\oplus_{k=0}^n M_{m_k}(\C)$ be a finite-dimensional C*-algebra for some $n \in \N$ and $m_k \in \N \setminus \{0\}$ for $k \in \{0, \ldots, n\}$.  
Fix $k \in \{0, \ldots, n\}$ and $p,q \in \{1, \ldots, m_k\}$.  Let $e_{k,(p,q)} \in \A$ denote the {\em matrix unit} such that $ e_{k,(p,q)}=\left(a^1,\ldots, a^n \right)\in \A$ and $a^j=0$ for $j \in \{0, \ldots, n\} \setminus \{k\}$ and $a^k$ is the matrix in $M_{m_k}(\C)$ that is $1$ in the $p$-row, $q$-column entry and $0$ elsewhere.  We will use $e_{k,(p,q)}$ to denote this element in $\A$ as well as its projection onto $M_{m_k}(\C)$, which is $a^k$.
\end{definition}

\begin{notation}\label{fd-states-notation}
Let $(X, \mathsf{d}_X)$ be a compact metric space.  Assume that $\A$ is finite-dimensional and that there exist $n \in \N$ and $m_k \in \N \setminus \{0\}$ for $k \in \{0, \ldots, n\}$ such that $\A=\oplus_{k=0}^n M_{m_k}(\C)$. Let $ \A \otimes C(X)$ denote the C*-algebra formed over the algebraic tensor product of $\A$ and $C(X)$. By \cite[Proposition 2.4.2 and Proposition 3.6.12]{Brown-Ozawa}, we do not need to decorate the tensor (see the proof of Lemma (\ref{pure-state-lemma}) for details).  The map:
\begin{equation*}
\pi_\otimes^{\A,X}  : a\in  C(X, \A) \longmapsto \sum_{k=0}^n \sum_{p,q=1}^{m_k} e_{k,(p,q)} \otimes a^k_{p,q} \in \A \otimes C(X)
\end{equation*}
is the canoncial *-isomorphism, where $a^k_{p,q} \in C(X)$ and for each $x\in X$, the element $a^k_{p,q}(x) \in \C$ is the projection of $a(x)$ onto the $p$-row, $q$-column coordinate of $M_{m_k}(\C)$  (continuity of $a^k_{p,q}$ follows by finite-dimensionality and definiton of product topology).  That $\pi_\otimes^{\A,X}$ is a *-isomorphism follows from the observation that it is the inverse of the map $T$ used in the proof of Lemma (\ref{pure-state-lemma}).

For $\mu \in \StateSpace(\A), \nu \in \StateSpace(C(X))$, let $\mu \otimes \nu : \A \otimes C(X) \longrightarrow \C$ denote the state on $\A \otimes C(X)$ such that on elementary tensors $(\mu \otimes \nu) (a \otimes f)=\mu(a)\nu(f)$ for all $a \in \A, f \in C(X)$ \cite[Proposition 3.4.6]{Brown-Ozawa}.  For $x \in X$, let:
\begin{equation*}
\delta_x : f \in C(X) \longmapsto f(x) \in \C
\end{equation*}
denote the Dirac point mass at $x$.  The map $\delta_x \in \StateSpace(C(X))$, and for a state $\mu \in \StateSpace(\A)$, denote: 
\begin{equation}\label{fd-states-eq}
\begin{split}
&\text{(1) } \ \ \ \ k_\mu^\A=\sum_{k=0}^n \sum_{p,q=1}^{m_k} \left\vert \mu\left(e_{k,(p,q)}\right) \right\vert,\\
&\text{(2) } \ \ \ \ \mu_x = \left(\mu \otimes \delta_x \right)\circ \pi_\otimes^{\A,X} \in \StateSpace(C(X, \A)), \text{ and }\\
&\text{(3) }\ \ \ \ \Delta_\mu^\A : x \in  (X, \mathsf{d}_X) \longmapsto \mu_x \in \StateSpace(C(X, \A)).
\end{split}
\end{equation}
\end{notation}

Now, it is easy to show that the map $\Delta_\mu^\A$ is a homeomorphism onto its image with respect to the weak* topology for any state $\mu \in \StateSpace(\A)$ (we, in fact, verify this in Proposition (\ref{any-state-lip-prop})).  However, when $X$ is metric we will show that $\Delta_\mu^\A$ captures some metric structure when $\mu \in \StateSpace(\A)$ is any state by way of the Monge-Kantorovich metric (quantum metric) induced by the seminorms of Definition (\ref{homog-lipschitz-def}), and for particular choices of states, we will be capable of isometrically embedding $X$ into $\StateSpace (C(X, \A))$, which thus generalizes the classical case of $\A=\C$ to the case  when $\A$ is finite-dimensional, where the case of $\A=\C$ was discuss in the introduction.  This next Proposition (\ref{any-state-lip-prop}) also shows that for any state $\mu$ and any of the seminorms of  Definition (\ref{homog-lipschitz-def}), we can embed $X$ homeomorphically into $\StateSpace(C(X,\A))$ with a $k$-Lipschitz map, where the constant $k$ depends on $\mu$, the matrix units of Defintion (\ref{matrix-units-def}), and the choice of norm on the finite-dimensional C*-algebra $\A$. First, we state the following lemma about a certain standard norm on the underlying vector space of a finite-dimensional C*-algebra, which will prove useful later as well.
\begin{lemma}\label{norm-equiv-lemma}
Fix $n \in \N$ and $m_k \in \N \setminus \{0\}$ for $k \in \{0, \ldots, n\}$.  Let $\A=\oplus_{k=0}^n M_{m_k}(\C)$ be a finite-dimensional C*-algebra.  Let $a=(a^k)_{k=0}^n \in \A$, where $a^k \in M_{m_k}(\C)$ for all $k \in \{0, \ldots, n\}$. Fix $k \in \{0, \ldots, n\}$ and let $a^k_{p,q}$ denote the $p$-row, $q$-column entry of the matrix $a^k \in M_{m_k}(\C)$.  Let $m_\A = \max_{k \in \{0, \ldots, n\}} m_k$.

 If  $\Vert \cdot \Vert_{\infty, \A}$ denotes the norm on $\A$ defined by $\Vert a \Vert_{\infty, \A}=\max_{k \in \{0, \ldots, n\}} \left\Vert a^k \right\Vert_{\infty}$, where $\left\Vert a^k \right\Vert_{\infty} = \max_{p,q\in \{1, \ldots, m_k\}} \left\vert \left(a^k\right)_{p,q} \right\vert$ for each $k \in \{0, \ldots, n\}$, then  this norm is equivalent to $\Vert \cdot \Vert_\A$  by:
\begin{equation*}
\frac{1}{m_\A}\Vert \cdot \Vert_\A \leq \Vert \cdot \Vert_{\infty, \A} \leq  \Vert \cdot \Vert_\A.
\end{equation*}
\end{lemma}
\begin{proof}
The conclusion follows  from the table of norm equivalences given in  \cite[page 314]{Horn}.
\end{proof}
\begin{remark}\label{max-norm-remark}
If, in Corollary (\ref{C(X)-cor}), we used the norm $\Vert \cdot \Vert_{\infty, \A}$ of Lemma (\ref{norm-equiv-lemma}) instead of $\Vert \cdot \Vert_\A$, then the map $\alg{c}_X$ would still be an isometry of the Lip-norms $\Lip_{\mathsf{d}_X}$ and  $\Lip_{\mathsf{d}_X}^{(\infty,\A),C(X)} $ by the same argument of the proof of Corollary (\ref{C(X)-cor})  since the norms $\Vert \cdot \Vert_{\infty, \A}$ and $\Vert \cdot \Vert_{\A}$  agree on scalars, but the associated  compact quantum metric space would not be Leibniz but $\left(m_\A,0\right)$-quasi-Leibniz by Theorem (\ref{fd-thm}) and Lemma (\ref{norm-equiv-lemma}).
\end{remark}
\begin{proposition}\label{any-state-lip-prop}
Let $(X, \mathsf{d}_X)$ be a compact metric space and let $\A=\oplus_{k=0}^n M_{m_k}(\C)$ be a finite dimensional C*-algebra, where such that $n \in \N$ and $m_k \in \N \setminus \{0\}$ for $k \in \{0, \ldots, n\}$.

Using Expression (\ref{fd-states-eq}), if $\mu \in \StateSpace(\A)$, then the map $
\Delta_\mu^\A$
is a homeomorphism onto its image with respect to the weak* topology on $\StateSpace(C(X, \A))$.

Furthermore, using notation from Definition (\ref{homog-lipschitz-def}) and Definition (\ref{Monge-Kantorovich-def}), if $\Vert \cdot \Vert_\mathsf{n}$ is a norm on $\A$ over $\R$ or $\C$ and $M,N>0$ such that $M\cdot \Vert \cdot \Vert_{\mathsf{n}} \leq \Vert \cdot \Vert_\A \leq N \cdot \Vert \cdot \Vert_{\mathsf{n}}$, then  for all $x,y \in X$, we have that:
\begin{equation*}
\Kantorovich{\Lip_{\mathsf{d}_X}^{(\mathsf{n}),q}}\left(\Delta_\mu^\A(x), \Delta_\mu^\A(y)\right) \leq N\cdot k_\mu^\A  \cdot \mathsf{d}_X (x,y),
\end{equation*}
and thus $\Delta_\mu^\A$ is $N\cdot  k_\mu^\A$-Lipschitz.
\end{proposition}
\begin{proof}
To show that $\Delta_\mu^\A$ is a homeomorphism onto its image, we do not need to first show the Lipscitz condition and we do not need that $X$ is  metric; however, we still rely on $X$ being Hausdorff.  Let $(x_\lambda)_{\lambda \in \Lambda} \subseteq X$ be a net that converges to $x \in X$.  Thus, for any $f \in C(X)$, we have that $\left(f\left(x_\lambda\right)_{\lambda \in \Lambda} \right) \subset \C$ converges to $f(x)\in \C$.  Let $ a\in C(X, \A)$.  Fix $\lambda \in \Lambda$:
\begin{equation}\label{state-eval-eq}
\begin{split}
\Delta_\mu^\A\left(x_\lambda\right) (a) &=  \mu_{x_\lambda}(a)=\left(\mu \otimes \delta_{x_\lambda}\right) \left(\sum_{k=0}^n \sum_{p,q=1}^{m_k} e_{k,(p,q)} \otimes a^k_{p,q}   \right)\\
&=\sum_{k=0}^n \sum_{p,q=1}^{m_k}\mu\left( e_{k,(p,q)}\right)\delta_{x_\lambda}\left(a^k_{p,q} \right)=\sum_{k=0}^n \sum_{p,q=1}^{m_k}\mu\left( e_{k,(p,q)}\right)a^k_{p,q} \left( x_{\lambda}\right).
\end{split}
\end{equation}
Since this sum is finite and $\lambda \in \Lambda$ was arbitrary, we have that  $\left(\Delta_\mu^\A\left(x_\lambda\right) (a)\right)_{\lambda\in \Lambda} \subset \C$ converges to  $\Delta_\mu^\A\left(x\right) (a)\in \C$ .  Thus, the net $\left(\Delta_\mu^\A\left(x_\lambda\right)\right)_{\lambda\in \Lambda} \subset \StateSpace(C(X, \A))$ converges to $\Delta_\mu^\A\left(x\right) \in \StateSpace(C(X, \A))$ with respect to the weak* topology since $a \in C(X, \A)$ was arbitrary.  Hence, since $X$ is compact and the weak* topology is Hausdorff, we have that $\Delta_\mu^\A$ is a homeomorphism onto its image.

Next, fix $x,y \in X$ and let $a \in \dom{\Lip_{\mathsf{d}_X}^{(\mathsf{n}),q}}$ such that $\Lip_{\mathsf{d}_X}^{(\mathsf{n}),q}(a) \leq 1$.  Using Lemma (\ref{norm-equiv-lemma}), we have:
\begin{equation*}
\frac{1}{N} \Vert a(x)- a(y) \Vert_{\infty, \A} \leq \frac{1}{N} \Vert a(x)- a(y) \Vert_{ \A}  \leq \Vert a(x)-a(y) \Vert_\mathsf{n} \leq \mathsf{d}_X(x,y),
\end{equation*}
which provides that $\Vert a(x)- a(y) \Vert_{\infty, \A} \leq N \mathsf{d}_X (x,y)$.  Hence, following Expression (\ref{state-eval-eq}), we gather that:
\begin{align*}
&\left\vert \Delta^\A_\mu(x)(a) - \Delta^\A_\mu(y)(a) \right\vert  =\left\vert \mu_x(a) - \mu_y(a) \right\vert \\
&  =\left\vert \sum_{k=0}^n \sum_{p,q=1}^{m_k}\mu\left( e_{k,(p,q)}\right)a^k_{p,q} \left( x\right) - \sum_{k=0}^n \sum_{p,q=1}^{m_k}\mu\left( e_{k,(p,q)}\right)a^k_{p,q} \left( y\right)  \right\vert\\
& \leq \sum_{k=0}^n \sum_{p,q=1}^{m_k} \left\vert   \mu\left( e_{k,(p,q)}\right)\right\vert \left\vert a^k_{p,q} \left( x\right) - a^k_{p,q} \left( y\right)  \right\vert \\
& \leq N \left(\sum_{k=0}^n \sum_{p,q=1}^{m_k} \left\vert   \mu\left( e_{k,(p,q)}\right)\right\vert\right) \mathsf{d}_X (x,y) = Nk^\A_\mu \mathsf{d}_X (x,y),
\end{align*}
which implies that $\Kantorovich{\Lip_{\mathsf{d}_X}^{(\mathsf{n}),q}}\left(\Delta_\mu^\A(x), \Delta_\mu^\A(y)\right) \leq N k_\mu^\A \mathsf{d}_X (x,y)$  for all $x,y \in X$ by Definition (\ref{Monge-Kantorovich-def}), and the proof is complete.
\end{proof}

In order to produce bi-Lipschitz maps and isometries, we want to know more information about the expression $k_\mu^\A$, which requires us to focus our attention on particular states.  We will consider the family of tracial states of $\A$.   We note that the following theorem does include all tracial states on $\A$ by \cite[Example IV.5.4]{Davidson}.

\begin{theorem}\label{trace-bi-lip-thm}
Let $(X, \mathsf{d}_X)$ be a compact metric space and let  $\A=\oplus_{k=0}^n M_{m_k}(\C)$ be a finite dimensional C*-algebra, where such that $n \in \N$ and $m_k \in \N \setminus \{0\}$ for $k \in \{0, \ldots, n\}$. Let $\Vert \cdot \Vert_\mathsf{n}$ be a norm on $\A$ over $\R$ or $\C$ and let $M, N>0$ such that $M\cdot \Vert \cdot \Vert_\mathsf{n}\leq \Vert \cdot \Vert_\A \leq N\cdot \Vert \cdot \Vert_\mathsf{n}$ and let $\mu \in \StateSpace(C(X, \A))$.  Let $v=(v_0, \ldots, v_n) \in \R^{n+1}$ such that $v_k \in [0,1]$ for all $k \in \{0, \ldots, n\}$ and $1=\sum_{k=0}^n v_k$.

If we let $\mathsf{tr}_v^\A= \oplus_{k=0}^n v_k\mathsf{tr}_{m_k} : \A \longrightarrow \C$, where $\mathsf{tr}_{m_k}=\frac{1}{m_k}\mathsf{Tr}$ is the unique tracial state on $M_{m_k}(\C)$ for all $k \in \{0, \ldots, n\}$, then using notation from Definition (\ref{homog-lipschitz-def}), Definition (\ref{Monge-Kantorovich-def}), and Expression (\ref{fd-states-eq}), we have that: 
\begin{equation*}\Delta_{\mathsf{tr}_v^\A}^\A: \left(X, \mathsf{d}_X\right) \longrightarrow \left(\StateSpace(C(X, \A))\Kantorovich{\Lip_{\mathsf{d}_X}^{(\mathsf{n}),q}}\right), 
\end{equation*} is bi-Lipschitz, and in particular:
\begin{enumerate}
\item if $q=C(X)$, then for all $x,y \in X$:
\begin{equation*}
M\cdot \mathsf{d}_X(x,y) \leq \Kantorovich{\Lip_{\mathsf{d}_X}^{(\mathsf{n}),C(X)}}\left(\Delta_{\mathsf{tr}_v^\A}^\A(x), \Delta_{\mathsf{tr}_v^\A}^\A(y)\right) \leq N \cdot \mathsf{d}_X (x,y);
\end{equation*}
\item if $q=\C$, then:
\begin{enumerate} 
\item if $M\cdot \diam{X}{\mathsf{d}_X}\leq 1$, then for all $x,y \in X$:
 \begin{equation*}
M\cdot \mathsf{d}_X(x,y) \leq \Kantorovich{\Lip_{\mathsf{d}_X}^{(\mathsf{n}),\C}}\left(\Delta_{\mathsf{tr}_v^\A}^\A(x), \Delta_{\mathsf{tr}_v^\A}^\A(y)\right) \leq N \cdot  \mathsf{d}_X (x,y);
\end{equation*}
\item if $M\cdot \diam{X}{\mathsf{d}_X}>1$, then for all $x,y \in X$:
\begin{equation*}
\frac{1}{\diam{X}{\mathsf{d}_X}} \cdot \mathsf{d}_X(x,y) \leq \Kantorovich{\Lip_{\mathsf{d}_X}^{(\mathsf{n}),\C}}\left(\Delta_{\mathsf{tr}_v^\A}^\A(x), \Delta_{\mathsf{tr}_v^\A}^\A(y)\right) \leq N  \cdot \mathsf{d}_X (x,y);
\end{equation*}
\end{enumerate}
\item if $\mu \in \StateSpace(C(X, \A))$ and $q=\mu$, then:
\begin{enumerate} 
\item if $2M\cdot \diam{X}{\mathsf{d}_X}\leq 1$, then for all $x,y \in X$:
\begin{equation*}
M\cdot \mathsf{d}_X(x,y) \leq \Kantorovich{\Lip_{\mathsf{d}_X}^{(\mathsf{n}),\mu}}\left(\Delta_{\mathsf{tr}_v^\A}^\A(x), \Delta_{\mathsf{tr}_v^\A}^\A(y)\right) \leq N\cdot   \mathsf{d}_X (x,y);
\end{equation*}
\item if $2M\cdot \diam{X}{\mathsf{d}_X}>1$, then for all $x,y \in X$:
\begin{equation*}
\frac{1}{2\cdot \diam{X}{\mathsf{d}_X}}\cdot  \mathsf{d}_X(x,y) \leq\Kantorovich{\Lip_{\mathsf{d}_X}^{(\mathsf{n}),\mu}}\left(\Delta_{\mathsf{tr}_v^\A}^\A(x), \Delta_{\mathsf{tr}_v^\A}^\A(y)\right) \leq N \cdot  \mathsf{d}_X (x,y).
\end{equation*}
\end{enumerate}
\end{enumerate}
\end{theorem}
\begin{proof}
We begin by calculating $k_{\mathsf{tr}_v^\A}^\A$, which is independent of the choice of $q$:
\begin{align*}
k_{\mathsf{tr}_v^\A}^\A& =\sum_{k=0}^n \sum_{p,q=1}^{m_k} \left\vert   \mathsf{tr}_v^\A\left( e_{k,(p,q)}\right)\right\vert  =\sum_{k=0}^n \frac{v_k}{m_k} \sum_{p,q=1}^{m_k} \left\vert \mathsf{tr}_{m_k}\left( e_{k,(p,q)}\right)\right\vert \\
& = \sum_{k=0}^n \frac{v_k}{m_k} \sum_{p,q=1, p=q}^{m_k} 1=\sum_{k=0}^n \frac{v_k}{m_k}m_k = \sum_{k=0}^n v_k=1
\end{align*}
by definition of the trace of a matrix $\mathsf{Tr}$ and matrix units (Definition (\ref{matrix-units-def})), where in the second line we view each $e_{k,(p,q)}$ as an element of $M_{m_k}(\C)$. This establishes the upper bounds for statements (1), (2), and (3)

For statement (1) and the lower bound, assume $q=C(X)$. Fix $x,y \in X$.  Define the function \begin{equation}y_{\mathsf{d}_X}: z \in X \longmapsto \mathsf{d}_X (y,z) \in \R
\end{equation} and note that $y_{\mathsf{d}_X} \in \sa{C(X)}$.  Next, define $Y_{\mathsf{d}_X}(z)= y_{\mathsf{d}_X}(z) 1_\A$ for all $z \in X$, and thus  $Y_{\mathsf{d}_X} \in \sa{C(X, \C1_\A)}$.   We thus have for $w,z \in X$:
\begin{equation}\label{metric-fcn-lip-ineq}
\begin{split}
\left\Vert MY_{\mathsf{d}_X}(w)-MY_{\mathsf{d}_X}(z) \right\Vert_{\mathsf{n}}& = M\Vert Y_{\mathsf{d}_X}(w)-Y_{\mathsf{d}_X}(z) \Vert_{\mathsf{n}}\\
&\leq \left\Vert Y_{\mathsf{d}_X}(w)-Y_{\mathsf{d}_X}(z) \right\Vert_{\A}\\
& = \left\vert y_{\mathsf{d}_X}(w)-y_{\mathsf{d}_X}(z) \right\vert =\left\vert \mathsf{d}_X(y,w) - \mathsf{d}_X (y,z) \right\vert \leq \mathsf{d}_X(w,z).
\end{split}
\end{equation}
Hence $ MY_{\mathsf{d}_X} \in \dom{\Lip_{\mathsf{d}_X}^{(\mathsf{n}),C(X)}}$ with $\Lip_{\mathsf{d}_X}^{(\mathsf{n}),C(X)}\left(MY_{\mathsf{d}_X}\right) \leq 1$ since:
\begin{equation*}\left\Vert  MY_{\mathsf{d}_X} + C\left(X, \C1_\A\right) \right\Vert_{C(X, \A)/C\left(X, \C1_\A\right)}=0
\end{equation*} by construction. From Notation (\ref{fd-states-notation}), note further that $\pi_\otimes^{\A,X} \left(Y_{\mathsf{d}_X}\right)=1_\A \otimes y_{\mathsf{d}_X} \in \A \otimes C(X)$.  Therefore:
\begin{equation}\label{mk-metric-fcn-lip-eq}
\begin{split}
&\left\vert \Delta_{\mathsf{tr}_v^\A}^\A (x)\left(MY_{\mathsf{d}_X}\right)-\Delta_{\mathsf{tr}_v^\A}^\A (y)\left(MY_{\mathsf{d}_X}\right) \right\vert  =\left\vert  {\mathsf{tr}_v^\A}_x \left(MY_{\mathsf{d}_X}\right) -{\mathsf{tr}_v^\A}_y \left(MY_{\mathsf{d}_X}\right)  \right\vert \\
& = M\left\vert \left(\mathsf{tr}_v^\A \otimes \delta_x\right)\left( 1_\A \otimes  y_{\mathsf{d}_X}\right) - \left(\mathsf{tr}_v^\A \otimes \delta_y\right)\left( 1_\A \otimes  y_{\mathsf{d}_X}\right)\right\vert \\
& =M \left\vert \delta_x\left( y_{\mathsf{d}_X}\right) -\delta_y\left( y_{\mathsf{d}_X}\right) \right\vert =M\left\vert \mathsf{d}_X(y,x) - \mathsf{d}_X(y,y) \right\vert =  M\mathsf{d}_X(x,y).
\end{split}
\end{equation}
Thus $M\mathsf{d}_X(x,y) \leq \Kantorovich{\Lip_{\mathsf{d}_X}^{(\mathsf{n}),C(X)}}\left(\Delta_{\mathsf{tr}_v^\A}^\A(x), \Delta_{\mathsf{tr}_v^\A}^\A(y)\right)$ by Definition (\ref{Monge-Kantorovich-def}). 

For statement (2) and the lower bound, assume $q=\C$. First assume that $M\diam{X}{\mathsf{d}_X}\leq 1$. Fix $x,y \in X$. Now, consider $MY_{\mathsf{d}_X}$, which still satsifies Expression (\ref{metric-fcn-lip-ineq}).  However, in the quotient norm, we have:
\begin{equation}\label{metric-fcn-quo-ineq}
\begin{split}
\left\Vert MY_{\mathsf{d}_X} + \C1_{C(X, \A)} \right\Vert_{C(X, \A)/ \C1_{C(X, \A)}} & \leq \left\Vert MY_{\mathsf{d}_X} -M \diam{X}{\mathsf{d}_X}1_{C(X, \A)} \right\Vert_{C(X, \A)}\\
& \leq M\sup_{z \in X} \left\Vert \mathsf{d}_X(y,z)1_\A-\diam{X}{\mathsf{d}_X}1_\A \right\Vert_\A \\
& \leq M\sup_{z \in X} \left\vert \mathsf{d}_X(y,z)-\diam{X}{\mathsf{d}_X} \right\vert\\
& =  M\sup_{z \in X} \left( \diam{X}{\mathsf{d}_X}-\mathsf{d}_X(y,z) \right) \\
& \leq M \diam{X}{\mathsf{d}_X} \leq 1
\end{split}
\end{equation}
since $\diam{X}{\mathsf{d}_X} \geq \mathsf{d}_X(y,z)\geq 0$ for all $y,z \in X$.
Therefore, the assumption $\Lip_{\mathsf{d}_X}^{(\mathsf{n}),\C}\left(MY_{\mathsf{d}_X}\right) \leq 1$  and the argument of Expression (\ref{mk-metric-fcn-lip-eq}) applies, which proves (a) of statement (2).

For (b) of statement (2), assume that $M\diam{X}{\mathsf{d}_X}> 1$.  By Expression (\ref{metric-fcn-lip-ineq}) and since $\diam{X}{\mathsf{d}_X}>0$,  we have for all $w,z \in X$:
\begin{equation*}
\begin{split}
&\left\Vert \frac{1}{\diam{X}{\mathsf{d}_X}}Y_{\mathsf{d}_X}(w)-\frac{1}{\diam{X}{\mathsf{d}_X}}Y_{\mathsf{d}_X}(z) \right\Vert_{\mathsf{n}}\\
& = \frac{1}{M\diam{X}{\mathsf{d}_X}}\left\Vert MY_{\mathsf{d}_X}(w)-MY_{\mathsf{d}_X}(z) \right\Vert_{\mathsf{n}}\\
& \leq \frac{1}{M\diam{X}{\mathsf{d}_X}}\mathsf{d}_X(w,z) < \mathsf{d}_X(w,z).
\end{split}
\end{equation*}
By the same argument of Expression (\ref{metric-fcn-quo-ineq}), we have that:
\begin{equation*}
\left\Vert \frac{1}{\diam{X}{\mathsf{d}_X}}Y_{\mathsf{d}_X} + \C1_{C(X, \A)} \right\Vert_{C(X, \A)/ \C1_{C(X, \A)}}  \leq 1.
\end{equation*}
Following the process of Expression (\ref{mk-metric-fcn-lip-eq}), (2)(b) is proven. Statement (3) follows the same process as statement (2) along with the observation that for $y \in X$:
\begin{equation*}
\begin{split}
\left\Vert Y_{\mathsf{d}_X} - \mu\left(Y_{\mathsf{d}_X} \right)1_{C(X,\A)}\right\Vert_{C(X, \A)}
&  \leq 2\left\Vert Y_{\mathsf{d}_X}\right\Vert_{C(X, \A)}  =2 \sup_{z \in X } \left\Vert \mathsf{d}_X(y,z)1_\A\right\Vert_\A \\
& = 2 \sup_{z \in X } \left\vert \mathsf{d}_X(y,z)\right\vert\leq 2 \diam{X}{\mathsf{d}_X},
\end{split}
\end{equation*}
which completes the proof.
\end{proof}
\begin{remark}\label{diam-remark}
In the above Theorem (\ref{trace-bi-lip-thm}) and in  statements (2) and (3), the diameter of the metric space $X$ appears.  This is not too surprising because in Proposition (\ref{mk-metric-diam-lemma}) we do not see any relationship between the bound on the Monge-Kantorovich metric and the metric space $X$ in the cases of $q=\C$ and $q=\mu$.  Thus, it makes sense for the diameter to appear in these cases in Theorem (\ref{trace-bi-lip-thm}) to somewhat correct this discrepancy.
\end{remark}
In the next corollary, we will see that particular choices of norms on $\A$ in Theorem (\ref{trace-bi-lip-thm}) will provide us with isometries. Indeed:
\begin{corollary}\label{main-cor}
Let $(X, \mathsf{d}_X)$ be a compact metric space and let $\A$ be a finite dimensional C*-algebra such that  $\A=\oplus_{k=0}^n M_{m_k}(\C)$, where such that $n \in \N$ and $m_k \in \N \setminus \{0\}$ for $k \in \{0, \ldots, n\}$. Let $\mu \in \StateSpace(C(X, \A))$.  Let $v=(v_0, \ldots, v_n) \in \R^{n+1}$ such that $v_k \in [0,1]$ for all $k \in \{0, \ldots, n\}$ and $1=\sum_{k=0}^n v_k$.

If $\Vert \cdot \Vert_\mathsf{n}=\Vert \cdot \Vert_\A$ or $\Vert \cdot \Vert_\mathsf{n}=\Vert \cdot \Vert_{\infty, \A}$ of Lemma (\ref{norm-equiv-lemma}), then using notation from Theorem (\ref{trace-bi-lip-thm}), we have that: 
\begin{enumerate}
\item if $q=C(X)$, then $\Delta_{\mathsf{tr}_v^\A}^\A$ 
 is an isometry;
\item if $q=\C$, then:
\begin{enumerate} 
\item if $\diam{X}{\mathsf{d}_X}\leq 1$, then $\Delta_{\mathsf{tr}_v^\A}^\A$ 
 is an isometry;
\item if $\diam{X}{\mathsf{d}_X}>1$, then for all $x,y \in X$:
\begin{equation*}
\frac{1}{\diam{X}{\mathsf{d}_X}} \cdot \mathsf{d}_X(x,y) \leq \Kantorovich{\Lip_{\mathsf{d}_X}^{(\mathsf{n}),\C}}\left(\Delta_{\mathsf{tr}_v^\A}^\A(x), \Delta_{\mathsf{tr}_v^\A}^\A(y)\right) \leq  \mathsf{d}_X (x,y);
\end{equation*}
\end{enumerate}
\item if $\mu \in \StateSpace(C(X, \A))$ and $q=\mu$, then:
\begin{enumerate} 
\item if $2\cdot \diam{X}{\mathsf{d}_X}\leq 1$, then $\Delta_{\mathsf{tr}_v^\A}^\A$ 
 is an isometry;
\item if $2\cdot \diam{X}{\mathsf{d}_X}>1$, then for all $x,y \in X$:
\begin{equation*}
\frac{1}{2\cdot \diam{X}{\mathsf{d}_X}} \cdot \mathsf{d}_X(x,y) \leq\Kantorovich{\Lip_{\mathsf{d}_X}^{(\mathsf{n}),\mu}}\left(\Delta_{\mathsf{tr}_v^\A}^\A(x), \Delta_{\mathsf{tr}_v^\A}^\A(y)\right) \leq  \mathsf{d}_X (x,y).
\end{equation*}
\end{enumerate}
\end{enumerate}
In particular, if $X=\T$ is the circle  as either a subset of $\C$ or quotient space of $[0,1]$ with their standard metrics, then $\Delta_{\mathsf{tr}_v^\A}^\A$ is an isometry when $q=C(\T)$ or $q=\C$.
\end{corollary}
\begin{proof}
If $\Vert \cdot \Vert_{\mathsf{n}}=\Vert \cdot \Vert_\A$, then the conclusions follow immediately from Theorem (\ref{trace-bi-lip-thm}).  The case of $\Vert \cdot \Vert_{\mathsf{n}}=\Vert \cdot \Vert_{\infty,\A}$ follows from the same arguments in the proof of Theorem (\ref{trace-bi-lip-thm}) along with the observation that as $ Y_{\mathsf{d}_X}$ is scalar-valued: 
\begin{equation*}
\left\Vert Y_{\mathsf{d}_X}(w)-Y_{\mathsf{d}_X}(z) \right\Vert_{\infty,\A} 
 = \left\vert y_{\mathsf{d}_X}(w)-y_{\mathsf{d}_X}(z) \right\vert=\left\Vert Y_{\mathsf{d}_X}(w)-Y_{\mathsf{d}_X}(z) \right\Vert_{\A}
\end{equation*}
for $y \in X$ and all $w,z \in X$.
\end{proof}

\section{Convergence of standard homogeneous C*-algebras and finite-dimensional approximations}\label{converge-section}

Recalling Theorem (\ref{comm-propinquity-thm}), the topology induced by the Gromov-Hausdorff distance $\mathrm{GH}$ on compact metric spaces homeomorphically embeds into the quantum Gromov-Hausdorff propinquity topology by the following class map:
\begin{equation*}
\Gamma : (X,\mathsf{d}_X) \in (\mathrm{CMS}, \mathrm{GH})\longmapsto (C(X), \Lip_{\mathsf{d}_X}) \in (\mathrm{qLCQMS}, \qpropinquity{}), 
\end{equation*}
where $\mathrm{CMS}$ is the class of compact metric spaces and $\mathrm{qLCQMS}$ is the class of all quasi-Leibniz compact quantum metric spaces.  The bijection is with respect to the equivalence relations of isometry on the domain and quantum isometry on the codomain.  However, this result was true with Rieffel's quantum distance $\mathrm{dist}_q$ in 2000 \cite[Combine Proposition 4.7, Corollary 7.10, and Theorem 13.16]{Rieffel00}.  And, since 2000, the question of whether the continuity of $\Gamma$ extends to matrices over $C(X)$ has been left unanswered.  More formally, this question asks if for all $n \in \N \setminus \{0\}$, there exist Lip-norms $\Lip^n_X$ such that the following map
\begin{equation}\label{informal-cont-map-eq}
\Gamma^{M_n} : (X,\mathsf{d}_X) \in (\mathrm{CMS}, \mathrm{GH})\longmapsto \left(M_n(C(X)), \Lip^n_X\right) \in (\mathrm{qLCQMS}, \qpropinquity{}), 
\end{equation}
is continuous.  To be clear, this question involves not just continuity, but also asks what the Lip-norms $\Lip_n$ should be.  In this section, we finally answer this question by presenting Lip-norms from our work in previous sections for which continuity of this map does hold.  This will also produce finite-dimensional approximations as every compact metric space may be approximated in the Hausdorff distance by finite subsets. We note that in this section we prove the continiuity of the map in Expression (\ref{informal-cont-map-eq}) in slightly more generality and for spaces of the form $C(X,\A)$, where $(X, \mathsf{d}_X)$ is any compact metric and $\A$ is a fixed finite-dimensional C*-algebra and note that $C(X, M_n(\C))$ is canonically *-isomorphic to $M_n(C(X))$.

The idea of our proof relies on a result of E.J. McShane \cite[Theorem 1]{McShane34}, which informally states that one can extend a {\em real-valued} $K$-Lipschitz map defined on a subset of a metric space to a $K$-Lipschitz map on the whole space. The {\em real-valued} condition is key as this is not true in general for complex valued functions (see \cite[Example 1.5.7]{Weaver99}).  McShane's Theorem was used successfully by \Latremoliere \ in \cite[Theorem 6.6]{Latremoliere13} along with the fact that Lip-norms need only be defined on self-adjoints (real-valued functions in this case) to show continuity of $\Gamma$.  However, for $M_n(\C)$ with $n \geq 2$, the self-adjoint elements may have complex entries.  This has been the obstruction that has made this problem difficult to solve.  Our solution is provided by the simple observation that since Lip-norms need only be defined on self-adjoints, one need not used a complex seminorm defined on the whole C*-algebra and then  restricted to the self-adjoints.  We now define the seminorms which will provide our convergence results of this section, which are similar to the seminorms of (1) of Definition (\ref{homog-lipschitz-def}).

\begin{notation}\label{conv-lip-notation}
Let $(X, \mathsf{d}_X)$ be a compact metric space and $\A$ be a finite-dimensional C*-algebra such that there exists $n  \in \N$ and $m_l \in \N\setminus \{0\}$ for $l \in \{0, \ldots, n\}$ with $\A=\oplus_{l=0}^n M_{m_l}(\C)$.

For $\lambda \in \C$, let $\mathfrak{Re}(\lambda)$ denote its real part and $\mathfrak{Im}(\lambda)$ denote its imaginary part.  Define $|\lambda|_\infty=\max\{|\mathfrak{Re}(\lambda)|, |\mathfrak{Im}(\lambda)|\},$ which defines an $\R$-seminorm over $\C$, and we note that $\vert \lambda \vert_\infty \leq  \vert \lambda \vert \leq \sqrt{2} \vert \lambda \vert_\infty$ for all $\lambda \in \C$. For $a=(a^l_{j,k})_{l \in \{0,\ldots,n\},j,k\in \{1, \ldots, m_l\}} \in \sa{\A}$,  where $(a^l_{j,k})_{j,k\in \{1, \ldots, m_l\}} \in M_{m_l}(\C)$ for all $l \in \{0, \ldots,n\}$, define:
\begin{equation*}
\|a\|_{\infty, \A_\R}= \max_{l \in \{0,\ldots,n\},j,k\in \{1, \ldots, m_l\}} \left|a^l_{j,k}\right|_\infty,
\end{equation*} 
which is a norm over $\sa{\A}$, but is not a norm over $\A$ with respect to $\C$.

Define $m_\A =\max_{l \in \{0, \ldots,n\}} m_l$.  By Lemma (\ref{norm-equiv-lemma}), we have that:
\begin{equation}\label{real-equiv-eq}
\| a\|_{\infty, \A_\R} \leq \|a\|_\A \leq \sqrt{2} \cdot m_\A \cdot \| a\|_{\infty, \A_\R} 
\end{equation}
for all $a \in \sa{\A}$.  For all $a \in \sa{C(X, \A)}$, using notation from Definition (\ref{homog-lipschitz-def}), define:
\begin{equation*}
\Lip^{(\infty, \A_\R)}_{\mathsf{d}_X} (a)= \max\left\{l^{(\infty, \A_\R)}_{\mathsf{d}_X}(a), \inf_{r \in \R} \sup_{x \in X} \|a(x) - r1_\A\|_{\infty, \A_\R} \right\}.
\end{equation*}
\end{notation}
Using the previous section, we summarize some of the quantum metric structure associated to this Lip-norm.
\begin{proposition}\label{conv-lip-cqms-prop}
Let $(X, \mathsf{d}_X)$ be a compact metric space and $\A$ be a finite-dimensional C*-algebra such that there exists $n  \in \N$ and $m_l \in \N\setminus \{0\}$ for $l \in \{0, \ldots, n\}$ with $\A=\oplus_{l=0}^n M_{m_l}(\C)$. Let  $m_\A =\max_{l \in \{0, \ldots,n\}} m_l$.

Using Notation (\ref{conv-lip-notation}),  we have that $\left(C(X,\A), \Lip^{(\infty, \A_\R)}_{\mathsf{d}_X}\right)$ is a $(\sqrt{2}\cdot m_\A,0)$-quasi Leibniz compact quantum metric space for which: 
\begin{equation*}\diam{\StateSpace(C(X,\A))}{\Kantorovich{ \Lip^{(\infty, \A_\R)}_{\mathsf{d}_X}}}\leq 2\sqrt{2}\cdot m_\A.
\end{equation*}
\end{proposition}
\begin{proof}
For the quasi-Leibniz condition.  We first note that since $\C1_{C(X,\A)}$ is finite-dimensional, it satisifies best approximation in $C(X,\A)$.  That is, if $a \in \sa{\A}$, there exists $\lambda \in \C$ such that $\|a+\C1_{C(X,\A)}\|_{C(X,\A)/\C1_{C(X,\A)}}=\|a-\lambda\cdot 1_{C(X,\A)}\|_{C(X,\A)}$.  Furthermore, since $a\in \sa{\A}$, we can assume that $\lambda \in \R$ since $\frac{\lambda +\lambda^*}{2}$ is also a best approximation. In summary, we have that $\|a+\C1_{C(X,\A)}\|_{C(X,\A)/\C1_{C(X,\A)}}=\|a+\C1_{C(X,\A)}\|_{C(X,\A)/\R1_{C(X,\A)}}.$ But, by definition,
\begin{equation*}
\|a+\C1_{C(X,\A)}\|_{C(X,\A)/\C1_{C(X,\A)}}=\inf_{r \in \R} \|a-r1_{C(X,\A)}\|_{C(X,\A)}= \inf_{r \in \R} \sup_{x \in X} \|a(x)-r1_\A\|_\A.
\end{equation*}
Hence, by Expression (\ref{real-equiv-eq}), we have the equivalence:
\begin{equation}\label{quo-equiv-eq} 
\begin{split}\inf_{r \in \R} \sup_{x \in X} \|(\cdot)(x)-r1_\A\|_{\infty, \A_\R} &\leq \| (\cdot) +\C1_{C(X,\A)}\|_{C(X,\A)/\C1_{C(X,\A)}} \\
& \leq \sqrt{2} \cdot m_\A \cdot \inf_{r \in \R} \sup_{x \in X} \|(\cdot)(x)-r1_\A\|_{\infty, \A_\R}.
\end{split}
\end{equation}
However, as $\| (\cdot) +\C1_{C(X,\A)}\|_{C(X,\A)/\C1_{C(X,\A)}}$ is $(1,0)$-quasi-Leibniz by \cite[Theorem 3.1]{Rieffel11}, then $\inf_{r \in \R} \sup_{x \in X} \|(\cdot)(x)-r1_\A\|_{\infty, \A_\R}$ is $(\sqrt{2}\cdot m_\A,0)$-quasi Leibniz by the same argument of Proposition (\ref{Leibniz-homog-norm-prop}), which provides the desired quasi-Leibniz condition for $\Lip^{(\infty, \A_\R)}_{\mathsf{d}_X}$. 

Next, we note that the expression $\inf_{r \in \R} \sup_{x \in X} \|(\cdot)(x)-r1_\A\|_{\infty, \A_\R}$ is simply the quotien norm onto the scalars of $C(X, \sa{\A})$ equipped with the  supremum norm induced by $\| \cdot \|_{\infty, \A_\R}$, i.e. $\sup_{x \in X} \|a(x)\|_{\infty,\A_\R}$ for all $a \in C(X,\sa{\A})$.  Thus, as $\R1_{C(X,\A)}$ is a closed subspace of the Banach Space $C(X, \sa{\A})$ with respect to this norm, then we have that the kernel of $\Lip^{(\infty, \A_\R)}_{\mathsf{d}_X}$ is $\R1_{C(X,\A)}$.  The domain of $\Lip^{(\infty, \A_\R)}_{\mathsf{d}_X}$ is dense by Lemma (\ref{dom-dense-lemma}).  Expression (\ref{quo-equiv-eq}) produces the bound on the diameter of the Monge-Kantorovich metric by \cite[Proposition 1.6]{Rieffel98a}.  Thus, the fact that $\left(C(X,\A), \Lip^{(\infty, \A_\R)}_{\mathsf{d}_X}\right)$ is a $(\sqrt{2}\cdot m_\A,0)$-quasi Leibniz compact quantum metric space follows from Theorem (\ref{fd-thm}).
\end{proof}
\begin{remark}
Similar conclusions to Proposition (\ref{C(X)-prop}) and the results of Section (\ref{mk-metric-section}) can be made with respect to the Lip-norm $\Lip^{(\infty, \A_\R)}_{\mathsf{d}_X}$, but we do not need them here and do not list them for the purpose of presentation.
\end{remark}
We are ready to prove one of our main convergence results, which will show that our Lip-norm capitalizes on the real structure of $\sa{\A}$ in more ways than one.  Indeed, the second quantity in the expression of $\Lip^{(\infty, \A_\R)}_{\mathsf{d}_X}$ pertains to bounds on elements in $\sa{C(X,\A)}$.  In particular, we require a more powerful version of McShane's theorem in that we need our Lipschitz extensions to preserve upper and lower bounds as well.  
\begin{theorem}\label{conv-lip-conv-thm}
Let $\A$ be a finite-dimensional C*-algebra such that there exists $n  \in \N$ and $m_l \in \N\setminus \{0\}$ for $l \in \{0, \ldots, n\}$ with $\A=\oplus_{l=0}^n M_{m_l}(\C)$. Let  $m_\A =\max_{l \in \{0, \ldots,n\}} m_l$.

Using Notation (\ref{conv-lip-notation}), Theorem (\ref{conv-lip-cqms-prop}), and Convention (\ref{equiv-class-convention}), the class map:
\begin{equation*}
\Gamma^\A : (X,\mathsf{d}_X) \in (\mathrm{CMS}, \mathrm{GH}) \mapsto \left( C(X, \A), \Lip^{(\infty, \A_\R)}_{\mathsf{d}_X}\right) \in (\mathrm{qLCQMS}, \qpropinquity{})
\end{equation*}
is well-defined and continuous. 

In particular, we have for any compact metric spaces $(X, \mathsf{d}_X), (Y, \mathsf{d}_Y)$:
\begin{equation}\label{q-comm-eq}
\qpropinquity{}\left( \left( C(X, \A), \Lip^{(\infty, \A_\R)}_{\mathsf{d}_X}\right),  \left( C(Y, \A), \Lip^{(\infty, \A_\R)}_{\mathsf{d}_Y}\right)\right) \leq \sqrt{2} \cdot m_\A\cdot \mathrm{GH}\left(\left(X, \mathsf{d}_X\right), \left(Y, \mathsf{d}_Y\right)\right).
\end{equation}
\end{theorem}
\begin{proof}
Proving Inequality (\ref{q-comm-eq}) would prove both well-defined and continuity of $\Gamma^\A$ with respect to the stated equivalence relations of Convention (\ref{equiv-class-convention}).  This proof follows the general strategy of the proof of \cite[Theorem 6.6]{Latremoliere13}.

Let $\delta_{X,Y}=\mathrm{GH}\left(\left(X, \mathsf{d}_X\right), \left(Y, \mathsf{d}_Y\right)\right)\geq 0.$   Let $\varepsilon>0$.  By definition of the Gromov-Hausdorff distance \cite[Definition 7.3.10]{burago01}, there exists a metric space $(Z,\mathsf{d}_Z)$ and isometries (not necessarily surjective)  $f_X : (X, \mathsf{d}_X) \rightarrow (Z,\mathsf{d}_Z)$ and $f_Y : (Y, \mathsf{d}_Y) \rightarrow (Z,\mathsf{d}_Z)$ such that the Hausdorff distance $\mathsf{Haus}_{\mathsf{d}_Z}(f_X(X), f_Y(Y)) \leq \delta_{X,Y} +  \frac{\varepsilon}{8\sqrt{2}\cdot m_\A}$.  Now, define a metric $\mathsf{d}_{X\sqcup Y}$ on the disjoint union $X \sqcup Y$ by:
\begin{equation*}
 \mathsf{d}_{X\sqcup Y}(\alpha,\beta) =\begin{cases}
 \mathsf{d}_X(\alpha,\beta) & : \text{ if }\alpha,\beta\in X\\
  \mathsf{d}_Z(f_X(\alpha), f_Y(\beta))+\frac{\varepsilon}{8\sqrt{2}\cdot m_\A} & : \text{ if } \alpha \in X, \beta \in Y \\
  \mathsf{d}_Z(f_X(\beta), f_Y(\alpha))+\frac{\varepsilon}{8\sqrt{2}\cdot m_\A} & : \text{ if } \alpha \in Y, \beta \in X \\
  \mathsf{d}_Y(\alpha,\beta) & : \text{ if } \alpha,\beta \in Y.
 \end{cases}
 \end{equation*}
 Clearly $X$ and $Y$ embed isometrically into $\left(X\sqcup Y, \mathsf{d}_{X \sqcup Y}\right)$ with respect to their associated metrics and:
 \begin{equation}\label{Haus-approx-eq} \mathsf{Haus}_{\mathsf{d}_{X \sqcup Y}}(X, Y) \leq \delta_{X,Y} +  \frac{\varepsilon}{4\sqrt{2}\cdot m_\A}
 \end{equation}
 by \cite[Remark 7.3.2]{burago01}, where $X,Y$ are viewed as subspaces of $X \sqcup Y$.
 
 Define $W=\left\{ (x,y)\in X \times Y: \mathsf{d}_{X\sqcup Y}(x,y) \leq \delta_{X,Y} +  \frac{\varepsilon}{2\sqrt{2}\cdot m_\A}\right\}$.  By construction, $W$ is compact in the product topology on $X \times Y$, and thus we equip $W$ with this topology.  Now, fix $x \in X$.  By definition of the Hausdorff distance and Expression (\ref{Haus-approx-eq}), there exists $y \in Y$ such that $\mathsf{d}_{X\sqcup Y}(x,y) \leq \delta_{X,Y} +  \frac{\varepsilon}{4\sqrt{2}\cdot m_\A} +\frac{\varepsilon}{4\sqrt{2}\cdot m_\A}=\delta_{X,Y} +  \frac{\varepsilon}{2\sqrt{2}\cdot m_\A}$.  Thus, for every $x \in X$, there exists $y \in Y$ such that $(x,y) \in W$.  Similarly,  for every $y \in Y$, there exists $x \in X$ such that $(x,y) \in W$.  In particular, the canonical projections $\rho_X : (x,y) \in W \mapsto x \in X$ and  $\rho_Y : (x,y) \in W \mapsto x \in Y$ are surjections and are continuous by definition of the product topology.  Therefore, the two maps $\pi_X: f \in C(X) \mapsto f \circ \rho_X \in C(W)$ and $\pi_Y: f \in C(Y) \mapsto f \circ \rho_Y \in C(W)$ are unital *-monomorphisms.  Hence, , we define a unital *-monomorphism $\pi_X^\A : C(X,\A) \rightarrow C(W,\A)$, where for every $a=\left(a^l_{j,k}\right)_{l \in \{0,\ldots,n\}, j,k \in \{1, \ldots, m_l\}} \in C(X,\A)$:
 \begin{equation}
 \pi_X^\A\left( \left(a^l_{j,k}\right)_{l \in \{0,\ldots,n\}, j,k \in \{1, \ldots, m_l\}}\right) = \left(\pi_X \circ a^l_{j,k}\right)_{l \in \{0,\ldots,n\}, j,k \in \{1, \ldots, m_l\}} 
 \end{equation} and $\pi_Y^\A: C(Y,\A) \rightarrow C(W,\A)$ is defined in the same way.  Therefore, the tuple $\gamma_{X,Y, \varepsilon}= \left(C(W,\A), 1_{C(W,\A)}, \pi_X^\A, \pi_Y^\A \right)$ defines a bridge from $C(X,\A)$ to $C(Y,\A)$ by \cite[Definition 3.6]{Latremoliere13}.  We claim that this bridge's length \cite[Definition 3.17]{Latremoliere13} is less than or equal to $\sqrt{2} \cdot m_\A \cdot \delta_{X,Y} + \frac{\varepsilon}{2}.$  First, we note that the bridge's height \cite[Definition 3.16]{Latremoliere13} is $0$ since the pivot is the identity $1_{C(W,\A)}$. Thus, we are left to find an upper bound for the bridge's reach \cite[Definition 3.14]{Latremoliere13}.
 
 Thus, let $a=\left(a^l_{j,k}\right)_{l \in \{0,\ldots,n\}, j,k \in \{1, \ldots, m_l\}} \in \sa{C(X,\A)}$ such that $\Lip_{\mathsf{d}_X}^{(\infty, \A_\R)}(a) \leq 1$. First, this implies that $\Lip_{\mathsf{d}_X}\left(\mathfrak{Re}\left(a^l_{j,k}\right)\right)\leq 1$ and $\Lip_{\mathsf{d}_X}\left(\mathfrak{Im}\left(a^l_{j,k}\right)\right)\leq 1$ for all $l \in \{0,\ldots,n\}, j,k \in \{1, \ldots, m_l\}.$  For all $l \in \{0,\ldots,n\},$ $ j,k \in \{1, \ldots, m_l\}$, let:
 \begin{equation*}
 \widehat{\mathfrak{Re}\left(a^l_{j,k}\right)}, \widehat{\mathfrak{Im}\left(a^l_{j,k}\right)}: X \sqcup Y \rightarrow \R
 \end{equation*}
 denote the Lipschitz-constant preserving extensions of $\mathfrak{Re}\left(a^l_{j,k}\right), \mathfrak{Im}\left(a^l_{j,k}\right),$ respectively, constructed in \cite[Theorem 1 and Corollary 2]{McShane34} with respect to $\mathsf{d}_{X\sqcup Y}$ that have the same greatest lower  bounds and least upper bounds of $\mathfrak{Re}\left(a^l_{j,k}\right), \mathfrak{Im}\left(a^l_{j,k}\right)$, respectively.
 
 Now, for every $l \in \{0, \ldots, n\}, j,k \in \{0, \ldots, m_l\}$, define:
  
\begin{equation*}b^l_{j,k}=\widehat{\mathfrak{Re}\left(a^l_{j,k}\right)}|_Y+i \widehat{\mathfrak{Im}\left(a^l_{j,k}\right)}|_Y,
\end{equation*} where $|_Y$ denotes restriction to $Y \hookrightarrow X \sqcup Y$. And, note that by construction  $b=\left(b^l_{j,k}\right)_{l \in \{0,\ldots,n\},j,k \in \{1, \ldots, m_l\}} \in \sa{C(Y, \A)}$ and $l^{(\infty, \A_\R)}_{\mathsf{d}_Y}(b) \leq 1$.
 
 Next, by the proof of Proposition (\ref{conv-lip-cqms-prop}) and the fact  $\R1_{C(X,\A)}$ is finite-dimensional and thus satisfies best approximation, there exists  $k_a \in \R$ such that:
 \begin{equation}
 \sup_{x \in X} \|a(x)-k_a1_\A\|_{\infty, \A_\R}=\inf_{r \in \R} \sup_{x \in X} \|a(x)-r1_\A\|_{\infty, \A_\R}\leq 1
 \end{equation}
 since $\Lip_{\mathsf{d}_X}^{(\infty, \A_\R)}(a) \leq 1$.  However, as $a \in \sa{C(X,\A)}$, we have that $\mathfrak{Im}\left(a^l_{j,j}\right)=0$ and thus $\mathfrak{Im}\left(b^l_{j,j}\right)=0$ by construction, which implies $b^l_{j,j}=\widehat{\mathfrak{Re}\left(a^l_{j,j}\right)}|_Y$ for all $l \in \{0, \ldots, n\}$ and $j \in \{1, \ldots, m_l\}$.  Therefore, if $z \in Y$, then:
 \begin{equation}\label{same-bound-ineq}
 -1 \leq \inf_{y \in Y} \left\{ \widehat{\mathfrak{Re}\left(a^l_{j,j}\right)}(y)-k_a\right\} \leq b^l_{j,j}(z)- k_a\leq \sup_{y \in Y} \left\{ \widehat{\mathfrak{Re}\left(a^l_{j,j}\right)}(y)-k_a \right\}\leq 1,
 \end{equation}
 and so $\sup_{y \in Y}| b^l_{j,j}(y)- k_a| \leq 1$ for all $l \in \{0, \ldots, n\}$ and $j \in \{1, \ldots, m_l\}$.  Similarly, we have that $\sup_{y \in Y} |b^l_{j,k}(y)|_\infty \leq 1$ for all $l \in \{0, \ldots, n\}$ and $j,k \in \{1, \ldots, m_l\}$.  Hence, as $k_a$ is only subtracted from diagonal entries, we conclude:
 \begin{equation*}
 \begin{split}
 \inf_{r \in \R} \sup_{y \in Y} \|b(y)-r1_\A\|_{\infty, \A_\R} \leq \sup_{y \in Y} \|b(y)-k_a1_\A\|_{\infty, \A_\R} \leq 1,
 \end{split}
 \end{equation*}
 which was all that remained to show that $\Lip^{(\infty, \A_\R)}_{\mathsf{d}_Y}(b) \leq 1$.  
 
 Now, let $(x,y)\in W$. Fix $l\in \{0, \ldots, n\},j,k\in \{1, \ldots, m_l\}$.  We have:
 \begin{equation*}
 \begin{split}
 &\left|\left(\pi_X \circ a^l_{j,k}\right)(x,y)-  \left(\pi_Y \circ b^l_{j,k}\right)(x,y)\right|_\infty \\
 & =  \left| a^l_{j,k}(x)-  b^l_{j,k}(y)\right|_\infty \\
 & = \max\left\{ \left| \widehat{\mathfrak{Re}\left(a^l_{j,k}\right)}(x)-\widehat{\mathfrak{Re}\left(a^l_{j,k}\right)}(y) \right|, \left| \widehat{\mathfrak{Im}\left(a^l_{j,k}\right)}(x)-\widehat{\mathfrak{Im}\left(a^l_{j,k}\right)}(y) \right|\right\}\\
 & \leq \max\left\{ \mathsf{d}_{X \sqcup Y} (x,y),  \mathsf{d}_{X \sqcup Y} (x,y)\right\}\\
 & \leq  \delta_{X,Y} +  \frac{\varepsilon}{2\sqrt{2}\cdot m_\A}.
 \end{split}
 \end{equation*}
 Hence: 
 \begin{equation*}
 \begin{split}
 \sup_{(x,y) \in W} \left\| \left(\pi_X^\A \circ a\right)(x,y) - \left(\pi_Y^\A \circ b\right)(x,y)\right\|_{\infty, \A_\R} \leq  \delta_{X,Y} +  \frac{\varepsilon}{2\sqrt{2}\cdot m_\A}, 
 \end{split}
 \end{equation*}
and therefore, by Expression (\ref{real-equiv-eq}):
 \begin{equation*}
 \begin{split}
 \left\| \left(\pi_X^\A \circ a\right)1_{C(W,\A)}- 1_{C(W,\A)}\left(\pi_Y^\A \circ b\right)(x,y)\right\|_{C(W,\A)} \leq  \sqrt{2}\cdot m_\A \cdot \delta_{X,Y} +  \frac{\varepsilon}{2}.
 \end{split}
 \end{equation*}
 The argument is symmetric if we began with an element in $b\in \sa{C(Y,\A)}$ such that  $\Lip^{(\infty, \A_\R)}_{\mathsf{d}_Y}(b) \leq 1$.  Hence, the quantity  $\sqrt{2}\cdot m_\A \cdot \delta_{X,Y} +  \frac{\varepsilon}{2}$ is an upper bound for the bridge $\gamma_{X,Y,\varepsilon}$ by \cite[Definition 3.14]{Latremoliere13}, and thus so is its length. Thus, by Theorem-Definition (\ref{def-thm}), we have that:
 \begin{equation*}
 \qpropinquity{}\left(\left(C(X, \A), \Lip^{(\infty, \A_\R)}_{\mathsf{d}_X}\right), \left( C(Y, \A), \Lip^{(\infty, \A_\R)}_{\mathsf{d}_Y}\right) \right) \leq \sqrt{2}\cdot m_\A \cdot \delta_{X,Y} +  \frac{\varepsilon}{2},
 \end{equation*}
for all $\varepsilon>0$.  Thus:
\begin{equation*}
\begin{split}
 \qpropinquity{}\left(\left(C(X, \A), \Lip^{(\infty, \A_\R)}_{\mathsf{d}_X}\right), \left( C(Y, \A), \Lip^{(\infty, \A_\R)}_{\mathsf{d}_Y}\right) \right)&  \leq \sqrt{2}\cdot m_\A \cdot \delta_{X,Y} \\
 & =\sqrt{2}\cdot m_\A \cdot \mathrm{GH}\left(\left(X, \mathsf{d}_X\right), \left(Y, \mathsf{d}_Y\right)\right),
 \end{split}
 \end{equation*}
 which completes the proof.
\end{proof}

As a corollary, we show that we may find finite-dimensional approximations to $C(X,\A)$ when $X$ is compact metric and $\A$ is finite dimensional.  Thus, we provide many new examples of propinquity extending the notion of approximate finite-dimensionality since $C(X,\A)$ need not be an AF algebra in general.

\begin{corollary}\label{fd-approx-cor}
Let $\A$ be a finite-dimensional C*-algebra such that there exists $n  \in \N$ and $m_l \in \N\setminus \{0\}$ for $l \in \{0, \ldots, n\}$ with $\A=\oplus_{l=0}^n M_{m_l}(\C)$. Let  $m_\A =\max_{l \in \{0, \ldots,n\}} m_l$.

If $(X, \mathsf{d}_X)$ is a compact metric space, there there exists a sequence $(X_n)_{n \in \N}$ of finite subsets of $X$ such that:
\begin{equation*}
\lim_{n \to \infty} \qpropinquity{}\left( \left( C(X_n, \A), \Lip^{(\infty, \A_\R)}_{\mathsf{d}_X}\right),  \left( C(X, \A), \Lip^{(\infty, \A_\R)}_{\mathsf{d}_X}\right)\right)=0.
\end{equation*}
In particular, as $ C(X_n, \A)$ is finite dimensional for all $n \in \N$, there exists a sequence of finite-dimensional C*-algebras equipped with $(\sqrt{2}\cdot m_\A,0)$-quasi-Leibniz Lip-norms converging in propinquity to $C(X,\A)$ equipped with a  $(\sqrt{2}\cdot m_\A,0)$-quasi-Leibniz Lip-norm.
\end{corollary}
\begin{proof}
Since $X$ is compact metric and thus totally bounded.  There exists a sequence $(X_n)_{n \in \N}$ of finite-subsets of $X$,  such that  $\lim_{ n \to \infty} \mathsf{Haus}_{\mathsf{d}_X}(X_n, X)=0$.  As the Gromov-Hausdorff distance is bounded above by the Hausdorff distance on compact subsets of a fixed compact metric space, the conclusion follows by Theorem (\ref{conv-lip-conv-thm}).
\end{proof}

Now, on $C(X)$,  the classical Lipschitz seminorm $\Lip_{\mathsf{d}_X}$ differs from the Lip-norm $\Lip_{\mathsf{d}_X}^{(\infty, \C_\R)}$ of Notation (\ref{conv-lip-notation}) in general  due its inclusion of the quotient norm.  Thus, our setting does not recover the setting of  Theorem (\ref{comm-propinquity-thm}) in general. However, if we allow ourselves to focus on classes of compact metric spaces with a fixed upper bound on diameter and we further provide a slight adjustment to our Lip-norms $\Lip_{\mathsf{d}_X}^{(\infty, \A_\R)}$ with respect to this bound, then we can recover Theorem (\ref{comm-propinquity-thm}) when $\A=\C$ by Theorem (\ref{conv-lip-conv-thm}) on these particular class of compact metric spaces.

\begin{theorem}\label{compact-recover-homeo-thm}
Let $\A$ be a finite-dimensional C*-algebra such that there exists $n  \in \N$ and $m_l \in \N\setminus \{0\}$ for $l \in \{0, \ldots, n\}$ with $\A=\oplus_{l=0}^n M_{m_l}(\C)$. Let  $m_\A =\max_{l \in \{0, \ldots,n\}} m_l$.

Fix $K >0$.  Denote the class of compact metric spaces with diameter less than or equal to $K$ by $\mathrm{CMS}_K.$ 

For every $(X, \mathsf{d}_X) \in \mathrm{CMS}_K$, if we define for all $a \in \sa{C(X,\A)}:$
\begin{equation*}
\Lip^{(\infty, \A_\R)}_{\mathsf{d}_X,K}(a) =\max \left\{l^{(\infty, \A_\R)}_{\mathsf{d}_X}(a), \frac{2}{K}\cdot \inf_{r \in \R} \sup_{x \in X} \|a(x)-r1_\A\|_{\infty, \A_\R}\right\},
\end{equation*} 
then using Convention (\ref{equiv-class-convention}):
\begin{enumerate}
\item  $\left(C(X,\A), \Lip^{(\infty, \A_\R)}_{\mathsf{d}_X,K}\right)$ is a $(\sqrt{2}\cdot m_\A,0)$-quasi-Leibniz compact quantum metric space such that
\item $\diam{\StateSpace(C(X,\A))}{\Kantorovich{ \Lip^{(\infty, \A_\R)}_{\mathsf{d}_X,K}}}\leq K \cdot \sqrt{2}\cdot m_\A, $
\item the map $\Gamma^\A$ of Theorem (\ref{conv-lip-conv-thm}) is well-defined and continuous if $ \Lip^{(\infty, \A_\R)}_{\mathsf{d}_X}$ and $\mathrm{CMS}$ are replaced with $ \Lip^{(\infty, \A_\R)}_{\mathsf{d}_X,K}$ and $\mathrm{CMS}_K$, and
\item if $n=0$ and $m_0=1$ so that $\A=\C$, then $\Lip^{(\infty, \C_\R)}_{\mathsf{d}_X,K}=\Lip_{\mathsf{d}_X}$ on $\sa{C(X)}$, and thus the map $\Gamma^\C$ of Theorem (\ref{conv-lip-conv-thm}) with, $ \Lip^{(\infty, \C_\R)}_{\mathsf{d}_X}$ and $\mathrm{CMS}$  replaced with $ \Lip^{(\infty, \C_\R)}_{\mathsf{d}_X,K}$ and $\mathrm{CMS}_K$, is a homeomorphism onto its image.
\end{enumerate}

In particular, if $\mathcal{K}$ is any compact class of compact metric spaces with respect to the Gromov-Hausdorff distance topology, then (1)-(4) are true for $\mathrm{CMS}_K$ replaced with $\mathcal{K}$, wherever $\mathrm{CMS}_K$ appears and the $K>0$ used for the Lip-norms is any fixed bound on the diameter of all compact metric spaces in $\mathcal{K}$.
\end{theorem} 
\begin{proof}
The proof of (1) and (2) follow from the same methods of Section (\ref{cqms-section}).  The proof of (3) follows from the proof of Theorem (\ref{conv-lip-conv-thm}) along with the fact that $K$ is fixed. Thus, (4) remains.

We note that $l^{(\infty, \C_\R)}_{\mathsf{d}_X}=\Lip_{\mathsf{d}_X}$ on $\sa{C(X)}$. Thus, to show $\Lip^{(\infty, \C_\R)}_{\mathsf{d}_X,K}=\Lip_{\mathsf{d}_X}$ on $\sa{C(X)}$, we only need to verify that $\frac{2}{K}\cdot \inf_{r \in \R} \sup_{x \in X} \|f(x)-r1_\C\|_{\infty, \C_\R} \leq \Lip_{\mathsf{d}_X}(f)$ for all $f \in \sa{C(X)}$. However, in this setting,  the quantity $\inf_{r \in \R} \sup_{x \in X} \|(\cdot)(x)-r1_\C\|_{\infty, \C_\R}$ on $\sa{C(X)}$ is simply the quotien norm with  respect $\C1_{C(X)}$ and the supremum norm with respect to absolute value.  
Therefore, for all $ f\in \sa{C(X)}$, we have that:
\begin{equation}\label{quo-lip-ineq}
\frac{2}{K} \cdot \inf_{r \in \R} \sup_{x \in X} \|f(x)-r1_\C\|_{\infty, \C_\R}
= \frac{2}{K} \cdot \frac{|f(x_m)-f(x_M)|}{2},
\end{equation}
where $x_m$ achieves the minumum of $f$ on $X$ and $x_M$ achieves the maximum of $f$ on $X$. Now, if $\diam{X}{\mathsf{d}_X}=0$, then $C(X)=\C$, and so any Lip-norm is the $0$-seminorm on $\sa{C(X)}$, which completes this case.  So, for the remainder of the proof, we assume that $\diam{X}{\mathsf{d}_X}>0.$  Thus:
\begin{equation*}
\begin{split}
\frac{2}{K} \cdot  \frac{|f(x_m)-f(x_M)|}{2} & \leq  \sup_{x,y \in X} \left\{ \frac{|f(x)-f(y)|}{K}\right\} \leq   \sup_{x,y \in X} \left\{ \frac{|f(x)-f(y)|}{\diam{X}{\mathsf{d}_X}}\right\}\\
& =  \sup_{x,y \in X, x \neq y} \left\{ \frac{|f(x)-f(y)|}{\diam{X}{\mathsf{d}_X}}\right\} \\
& \leq  \sup_{x,y \in X, x \neq y} \left\{ \frac{|f(x)-f(y)|}{\mathsf{d}_X(x,y)}\right\} = \Lip_{\mathsf{d}_X}(f).
\end{split}
\end{equation*}
 Thus, it must be the case that $\frac{2}{K} \cdot \inf_{r \in \R} \sup_{x \in X} \|f(x)-r1_\C\|_{\infty, \C_\R} \leq \Lip_{\mathsf{d}_X}(f)$  by Expression (\ref{quo-lip-ineq}),  which completes the proof of (4).

For the remaining statement, we simply note that any compact class of compact metric spaces with respect to the Gromov-Hausdorff distance topology is a subclass of $\mathrm{CMS}_K$ for come $K>0$ by the Gromov Compactness Theorem \cite[Theorem 7.4.15]{burago01}.
\end{proof}
\begin{remark}
Similar conclusions to Proposition (\ref{C(X)-prop}) and the results of Section (\ref{mk-metric-section}) can be made with respect to the Lip-norm $\Lip^{(\infty, \A_\R)}_{\mathsf{d}_X,K}$ of Theorem (\ref{compact-recover-homeo-thm}), but we do not need them here and do not list them for the purpose of presentation.

Furthermore, we note that (1), (2), and (4) of Theorem (\ref{compact-recover-homeo-thm}) would still hold true if $K$ were replaced by $\diam{X}{\mathsf{d}_X}$.  However, for (3), the $K$ is used in a non-trivial yet subtle way since we must compare these spaces in propinquity and if $K$ were allowed to vary then the method used involving Expression (\ref{same-bound-ineq}) would fail.  And, the importance of (3) is to show that we truly are extending  the continuity of the map $\Gamma$ of Theorem (\ref{comm-propinquity-thm}) to the matricial case on the subclass $\mathrm{CMS}_K$.
\end{remark}
\bibliographystyle{plain}
\bibliography{thesis-a}

 \vfill 
\end{document}